\documentclass[12pt]{article}
\usepackage{amsfonts}
\usepackage{amsmath, amssymb, amsthm}
\usepackage{enumerate}
\usepackage{epsfig}
\usepackage{verbatim}
\usepackage{graphicx}%
\usepackage{epsfig}

\newtheorem{theorem}{Theorem}

\newtheorem{lemma}[theorem]{Lemma}

\newcommand{\eps}{\varepsilon}

\begin{document}

\title{Analytic regularity for a fourth-order singularly perturbed boundary balue problem with two small parameters
}
\author{I. Sykopetritou and C. Xenophontos\thanks{
Corresponding author. Email: xenophontos@ucy.ac.cy} \\
%EndAName
Department of Mathematics and Statistics \\
University of Cyprus \\
P.O. Box 20537 \\
1678 Nicosia \\
Cyprus}
\maketitle

\begin{abstract}
We consider a fourth order singularly perturbed boundary value problem with
two small parameters, in one dimension, under the assumption of analytic input data. We show that the solution 
may be decomposed into a smooth part, two different width boundary layers, and a negligible remainder.
We provide estimates for arbitrary order derivatives of each term of the decomposition, which are explicit in
the differentiation order and the singular perturbation parameters, and are needed for
proving the convergence of high order numerical methods, such as the $p/hp$ versions
of the Finite Element Method. We also provide classical differentiability results, which show that the solution will be analytic, if the data are analytic, but negative powers of the singular perturbation parameter(s) show up once we start differentiating.
\end{abstract}

\textbf{Keywords}: fourth order singularly perturbed problem; boundary
layers; analytic regularity

\vspace{0.5cm}

\textbf{MSC2020}: 65L11

\section{Introduction}

\label{intro}

Singularly perturbed problems (SPPs) arise in applications when there are several scales present; for example, enzyme kinetics \cite{Frank}, geophysics \cite{Carrier}, chemical reactor theory \cite{LW}, lubrication theory \cite{CC}, liquid crystal studies \cite{AM}, and control theory \cite{KOS}, among others. 
As is well known, a main difficulty in these problems is the
presence of \emph{boundary layers} in the solution, which appear due to the fact
that the limiting problem (i.e. when the singular perturbation parameter(s)
tend to 0), is of different order than the original one, and the (extra)
boundary conditions can only be satisfied if the solution varies rapidly in
the vicinity of the boundary -- hence the name boundary layers, coined by
Prandtl in his seminal work ``On the motion of a fluid with very small viscosity", 
Third World Congress of Mathematicians, August 1903.

The numerical approximation of the solution to SPPs has been studied extensively over the last few decades (see,
e.g., the monographs \cite{mos}, \cite{morton}, \cite{rst} and the references
therein).  In most numerical methods, derivatives of the exact solution
appear in the error estimates, hence one should have a clear picture of how
these derivatives grow with respect to the singular perturbation
parameter(s). For low order numerical methods, such as Finite Differences
(FD) or the $h$ version of the Finite Element Method (FEM), derivatives up
to order 3 are usually sufficient. For high order methods such as the $p$ or $hp$
versions of the FEM, derivatives of arbitrary order are needed, thus knowing
how these behave with respect to the singular perturbation parameter(s) as
well as the differentiation order, is necessary. For second order SPPs with one parameter,
such derivative bounds may be found in, e.g., \cite{melenk}, and for SPPs with two parameters
in, e.g., \cite{SX}. For fourth order problems with one parameter, we refer to
\cite{PC, PCarXiv}. There are non-trivial differences between second and fourth order problems: for one, in second order one parameter problems the layers appear in the solution, while for fourth order problems they appear in the derivative of the solution.  Two parameter second order SPPs feature different width boundary layers at each endpoint, based on the relationship between the singular perturbation parameters. Unlike \cite{PC, PCarXiv}, the presence of two singular perturbation parameters produces two distinct boundary layer scales. Consequently, the recursive systems governing the layer terms become coupled through both $\eps_1$ and $\eps_2$, requiring substantially different analytic estimates.

We focus on a fourth order SPP with two singular perturbation parameters, multiplying the fourth and second derivatives (see  eq.~(\ref{de}) ahead). Under the assumption of analytic input data, we establish analytic
regularity estimates for each term in the decomposition of the solution into smooth and layers parts. These estimates form the basis for establishing the convergence order of high order FEMs (see, e.g.~\cite{XFS}); in fact it is \emph{needed} for the proof of the convergence rate for such methods.

The remainder of the paper is organized as follows: in Section \ref{model} we present the model problem and its regularity in terms of classical differentiability. Section \ref{asy} contains the asymptotic expansion for
the solution, and in Section \ref{erranal} we establish derivative bounds which are explicit in the differentiation order as well as the singular perturbation parameters. In Section \ref{nr} we provide an illustration through a numerical example. Finally in Section \ref{concl} we summarize our conclusions.

Throughout the article we will utilize the following notation: with $I\subset \mathbb{R}$ an interval with boundary $\partial I$ and measure $\left\vert I\right\vert $, we will denote by $C^{k}(I)$ the space
of continuous functions on $I$ with continuous derivatives up to order $k$.
We will use the usual Hilbert spaces $H^{k}\left( I\right)$, of functions on $I$
with weak derivatives of order $0,1,\ldots,k$ in $L^2(I)$, equipped
with the norm and seminorm,  $\left\Vert \cdot \right\Vert _{k,I}$
and $\left\vert \cdot \right\vert _{k,I}\,$, respectively. The usual $%
L^{2}(I)$ inner product will be denoted by $\left\langle \cdot ,\cdot
\right\rangle _{I}$, with the subscript omitted when there is no confusion.
We will also use the space 
\begin{equation*}
H_{0}^{2}\left( I\right) =\left\{ u\in H^{2}\left( I\right) :\left.
u\right\vert _{\partial I}=\left. u' \right\vert
_{\partial I}=0\right\} .
\end{equation*}%
The norm of the space $L^{\infty }(I)$ of essentially bounded functions is
denoted by $\Vert \cdot \Vert _{\infty ,I}$. Finally, the notation
\textquotedblleft $a\lesssim b$\textquotedblright\ means \textquotedblleft $%
a\leq Cb$\textquotedblright\ with $C$ being a generic positive constant,
independent of any discretization or singular perturbation parameters.

%%%%%%%%%%%%%%%%%%%%%%%%%%%%%%%%%%%%%%%%%%%%%%%%%%%%%%%%%%%%%%%%%%%%%%%%%%%%%%%%%%

\section{The model problem and the asymptotic expansion\label{model}}

We consider the following model problem: Find $u \in C^4(I) \cap C^1(\bar{I})$ such that 
\begin{eqnarray}
L_{\varepsilon_1, \varepsilon_2} u := \varepsilon _{1}^{2}u^{(4)}(x)-\varepsilon _{2}^{2}\left( b(x)u^{\prime
}(x)\right) ^{\prime }+c(x)u(x) &=&f(x)\ \text{in }I=\left( 0,1\right) ,
\label{de} \\
u(0)=u(1)=u^{\prime }(0)=u^{\prime }(1) &=&0  \label{bc}
\end{eqnarray}%
where $0<\varepsilon _{1},\varepsilon _{2}\leq 1$ are given parameters that
can approach zero and the functions $b,c,f$ are given and sufficiently
smooth. In particular, we assume that they are (real) analytic functions
satisfying, for some positive constants $\gamma _{f},\gamma _{c},\gamma _{b}$%
, independent of $\varepsilon _{1}$ and $\varepsilon _{2},$%
\begin{equation}
\left\Vert f^{(n)}\right\Vert _{\infty ,I}\lesssim n!\gamma
_{f}^{n},\left\Vert c^{(n)}\right\Vert _{\infty ,I}\lesssim n!\gamma
_{c}^{n},\left\Vert b^{(n)}\right\Vert _{\infty ,I}\lesssim n!\gamma
_{b}^{n}\;\forall \;n\in \mathbb{N}_{0}.  \label{analytic}
\end{equation}%
In addition, we assume that there exist positive constants $b_* ,c_*$,
independent of $\varepsilon _{1}$ and $\varepsilon _{2},$ such that $\forall
\;x\in \overline{I}$ 
\begin{equation}
b(x)\geq b_* >0,c(x)\geq c_*>0.  \label{data}
\end{equation}%
We will focus on the case 
\begin{equation}
\varepsilon _{1} \ll \varepsilon _{2}^{2},  \label{eps_relation}
\end{equation}%
since in the complementary case, we have a usual perturbation of
reaction-diffusion type with one-parameter, like the one studied in \cite{PCarXiv}.

It follows from the results of O'Malley  \cite{omalley, omalley2} that the solution $u$ to 
(\ref{de})--(\ref{bc}) satisfies $\Vert u \Vert_{\infty, I} \lesssim 1$. Moreover, $u$ can
be decomposed into a smooth part $S$, two boundary layer parts at the left
endpoint $E_{1}$ and $E_{2},$ and two boundary layer parts at the right endpoint $E_{3}$ and $E_{4}$, viz.%
\begin{equation*}
u=S+E_{1}+E_{2}+E_{3}+E_{4},
\end{equation*}%
with%
\begin{equation*}
\left\vert S^{(n)}(x)\right\vert \lesssim 1\;,\;\left\vert
E_{1}^{(n)}(x)\right\vert \lesssim \frac{1}{\varepsilon _{2}}\left( \frac{%
\varepsilon _{1}}{\varepsilon _{2}}\right) ^{1-n}e^{-\beta \varepsilon
_{2}x/\varepsilon _{1}}\;,\;\left\vert E_{2}^{(n)}(x)\right\vert \lesssim
\varepsilon _{2}^{-n}e^{-\beta x/\varepsilon _{2}},
\end{equation*}%
\begin{equation*}
\left\vert E_{3}^{(n)}(x)\right\vert \lesssim \frac{1}{\varepsilon _{2}}%
\left( \frac{\varepsilon _{1}}{\varepsilon _{2}}\right)
^{1-n}e^{-\beta \varepsilon _{2}(1-x)/\varepsilon _{1}}\;,\;\left\vert
E_{4}^{(n)}(x)\right\vert \lesssim \varepsilon
_{2}^{-n}e^{-\beta (1-x)/\varepsilon _{2}},
\end{equation*}%
for all $x\in \overline{I}$ and for $n=1,2,3$. This regularity result is sufficient
for proving convergence for a fixed order $h$ FEM, but not for a $p/hp$ FEM
-- a more refined regularity result is needed for the smooth part and bounds for
derivatives of arbitrary order should be explicit in terms of $\varepsilon_1, \varepsilon_2$,
as well as the order of differentiation. This is the main
goal of the article and will be achieved in the sections that follow.

We close this section with the following classical differentiability result.

\begin{theorem}
Let $u$ be the solution of (\ref{de})--(\ref{bc}) and assume (\ref{analytic}%
)--(\ref{eps_relation}) hold. Then, there exists a positive constant $K$,
independent of $\varepsilon _{1},\varepsilon _{2},$ such that for
all $n=1,2,\ldots $%
\begin{equation*}
\left\Vert u^{(n)}\right\Vert _{\infty ,I}\lesssim K^{n}\max \left\{
n^{n}, \frac{1}{\varepsilon_2} \left( \frac{\varepsilon _{1}}{\varepsilon _{2}}\right)
^{1-n}\right\} .
\end{equation*}
\end{theorem}

\begin{proof}
The proof is by induction on $n$ and follows \cite{melenk}. The desired result holds for 
$n=1,2,3$, by the results in \cite{omalley, omalley2}, so we assume it holds for $k \le n+3$
and show it for $k=n+4$.

Differentiating (\ref{de}) $n$ times gives

$$
\begin{aligned}
\varepsilon_1^2 u^{(n+4)} & =f^{(n)}+\varepsilon_2^2\left(b u^{\prime}\right)^{(n+1)}-(c u)^{(n)} \\
& =f^{(n)}+\sum_{\nu=0}^{n+1}\binom{n+1}{\nu} \varepsilon_2^2 b^{(\nu)} u^{(n+2-\nu)}-\sum_{\nu=0}^n\binom{n}{\nu} c^{(\nu)} u^{(n-\nu)} .
\end{aligned}
$$

By the induction hypothesis, we have

$$
\begin{aligned}
\varepsilon_1^2\left\|u^{(n+4)}\right\|_{\infty, I} \leq & \left\|f^{(n)}\right\|_{\infty, I}+ C \sum_{\nu=0}^{n+1}\binom{n+1}{\nu} \varepsilon_2^2 \nu!\gamma_b^\nu K^{n+2-\nu} \times \\
& \times \max \left\{(n+2-\nu)^{n+2-\nu}, \frac{1}{\varepsilon_2}\left(\frac{\varepsilon_1}{\varepsilon_2}\right)^{1-(n+2-\nu)}\right\}+ \\
& +C \sum_{\nu=0}^n\binom{n}{\nu} \nu!\gamma_c^\nu K^{n-\nu} \times \max \left\{(n-\nu)^{n-\nu}, \frac{1}{\varepsilon_2}\left(\frac{\varepsilon_1}{\varepsilon_2}\right)^{1-(n-\nu)}\right\} .
\end{aligned}
$$
Using the estimates below
$$
\begin{gathered}
\binom{n+1}{\nu} \nu!\max \left\{(n+2-\nu)^{n+2-\nu}, \frac{1}{\varepsilon_2}\left(\frac{\varepsilon_1}{\varepsilon_2}\right)^{1-(n+2-\nu)}\right\} \leq
 \max \left\{(n+2)^{n+2}, \frac{1}{\varepsilon_2}\left(\frac{\varepsilon_1}{\varepsilon_2}\right)^{1-(n+2)}\right\}, \\
\binom{n}{\nu} \nu!\max \left\{(n-\nu)^{n-\nu}, \frac{1}{\varepsilon_2}\left(\frac{\varepsilon_1}{\varepsilon_2}\right)^{1-(n-\nu)}\right\} \leq \max \left\{(n+2)^n, \frac{1}{\varepsilon_2}\left(\frac{\varepsilon_1}{\varepsilon_2}\right)^{1-n}\right\}, \\
\left\|f^{(n)}\right\|_{\infty, I} \leq C n!\gamma_f^n \leq C \gamma_f^n \max \left\{(n+2)^n, \frac{1}{\varepsilon_2}\left(\frac{\varepsilon_1}{\varepsilon_2}\right)^{1-n}\right\},
\end{gathered}
$$
which follow by elementary considerations, we obtain
$$
\begin{aligned}
\varepsilon_1^2\left\|u^{(n+4)}\right\|_{\infty, I} \leq & C \gamma_f^n \max \left\{(n+2)^n, \frac{1}{\varepsilon_2}\left(\frac{\varepsilon_1}{\varepsilon_2}\right)^{1-n}\right\}+ \\
& +C \varepsilon_2^2 K^{n+4} \max \left\{(n+2)^{n+2}, \frac{1}{\varepsilon_2}\left(\frac{\varepsilon_1}{\varepsilon_2}\right)^{1-(n+2)}\right\} \sum_{\nu=0}^{n+1} \frac{1}{K^2}\left(\frac{\gamma_b}{K}\right)^\nu+ \\
& +C K^{n+4} \max \left\{(n+2)^n, \frac{1}{\varepsilon_2}\left(\frac{\varepsilon_1}{\varepsilon_2}\right)^{1-n}\right\} \sum_{\nu=0}^n \frac{1}{K^4}\left(\frac{\gamma_c}{K}\right)^\nu \\
\leq & C K^{n+4} \max \left\{(n+2)^{n+2}, \frac{1}{\varepsilon_2}\left(\frac{\varepsilon_1}{\varepsilon_2}\right)^{1-n}+\varepsilon_2\left(\frac{\varepsilon_1}{\varepsilon_2}\right)^{-1-n}\right\} \times \\
& \times\left[\frac{1}{K^4}+\frac{1}{K^2} \frac{1}{\left(1-\gamma_b / K\right)}+\frac{1}{K^4} \frac{1}{\left(1-\gamma_c / K\right)}\right].
\end{aligned}
$$
We choose the constant $K>\max \left\{1, \gamma_f, \gamma_b, \gamma_c\right\}$
such that the expression in brackets above is bounded by 1, thus
$$
\varepsilon_1^2\left\|u^{(n+4)}\right\|_{\infty, I} \leq C K^{n+4} \max \left\{(n+2)^{n+2}, \frac{1}{\varepsilon_2}\left(\frac{\varepsilon_1}{\varepsilon_2}\right)^{1-n}+\varepsilon_2\left(\frac{\varepsilon_1}{\varepsilon_2}\right)^{-1-n}\right\}
$$
and dividing by $\varepsilon_1^2$ gives the result, since
$$
\frac{1}{\varepsilon_1^2} \frac{1}{\varepsilon_2}\left(\frac{\varepsilon_1}{\varepsilon_2}\right)^{1-n} \lesssim \frac{1}{\varepsilon_2}\left(\frac{\varepsilon_1}{\varepsilon_2}\right)^{1-(n+4)}
$$
and
$$
\frac{\varepsilon_2}{\varepsilon_1^2}\left(\frac{\varepsilon_1}{\varepsilon_2}\right)^{-1-n}=\frac{1}{\varepsilon_2}\left(\frac{\varepsilon_1}{\varepsilon_2}\right)^{1-(n+4)}
$$
due to $\varepsilon_1 \ll \varepsilon_2^2$.

\end{proof}

\subsection{The asymptotic expansion\label{asy}}

The method of matched asymptotic expansions \cite{OM} usually gives a more detailed picture of the regularity of $u$. To this end, we define the \emph{stretched variables} 
$$\tilde{x}=\frac{x}{\varepsilon _{2}},\bar{x}=\frac{x \varepsilon _{1}}{\varepsilon_{2}},\check{x}=\frac{1-x}{\varepsilon _{2}}, \hat{x}=\frac{(1-x) \varepsilon _{1}}{\varepsilon _{2}},
$$
and make the formal ansatz 
\begin{equation}
u\sim \sum_{i=0}^{\infty }\sum_{j=0}^{\infty }\varepsilon
_{2}^{i}(\varepsilon _{1}/\varepsilon _{2}^{2})^{j}\left( u_{i,j}(x)+\tilde{u%
}_{i,j}^{BL}(\tilde{x})+\bar{u}_{i,j}^{BL}(\bar{x})+\check{u}_{i,j}^{BL}(%
\check{x})+\hat{u}_{i,j}^{BL}(\hat{x})\right) ,  \label{c1}
\end{equation}%
with $u_{i,j},\tilde{u}_{i,j}^{BL},\bar{u}_{i,j}^{BL},\check{u}_{i,j}^{BL},%
\hat{u}_{i,j}^{BL}$ to be determined. Substituting (\ref{c1}) into (\ref{de}%
), separating the slow and fast variables, and equating like powers of $%
\varepsilon _{1},\varepsilon _{2}$, we get the following:

For $u_{i,j}$, i.e. for the (slow) variable $x$, we obtain
\begin{equation}
\left. 
\begin{array}{c}
u_{0,0} = \frac{f}{c} \\
u_{1,0}=0,u_{i,0}=\frac{\left( bu_{i-2,0}^{\prime }\right) ^{\prime }}{c}%
,i\geq 2 \\  
u_{i,1}=0,i\geq 0 \\
u_{0,j}=u_{1,j}=u_{2,j}=u_{3,j}=0,j\geq 2 \\
u_{i,j}=\frac{1}{c}\left\{ \left( bu_{i-2,j}^{\prime }\right) ^{\prime
}-u_{i-4,j-2}^{(4)}\right\} ,i\geq 4,j\geq 2%
\end{array}%
\right\},  \label{uij}
\end{equation}
with the details appearing in the Appendix.

For the boundary layer $\tilde{u}_{i,j}^{BL}$, i.e. for the (fast) variable $%
\tilde{x}=x/\varepsilon _{2}$, we first write $%
b(x)=b(\varepsilon _{2}x/\varepsilon _{2})=b(\varepsilon _{2}\tilde{x})$, and use a Taylor series to get
\begin{equation*}
b(\tilde{x})=\sum_{\ell =0}^{\infty }\varepsilon _{2}^{\ell } \tilde{b}_{\ell }(\tilde{x}),
\end{equation*}%
where $\tilde{b}_{\ell }(\tilde{x})=$ $\frac{b^{(\ell )}(0)}{\ell !}\tilde{x}^{\ell }$, and
similarly for $c$. The same procedure as above yields (see the Appendix for the details)

\begin{equation}
\left. 
\begin{array}{c}
-\tilde{b}_{0}\left(\tilde{u}^{BL}_{0,0}\right)^{\prime \prime }+\tilde{c}_{0}\tilde{u}^{BL}_{0,0}=0 \\
-\tilde{b}_{0}\left(\tilde{u}^{BL}_{i,0}\right)^{\prime \prime }+\tilde{c}_{0}\tilde{u}^{BL}_{i,0}=%
\sum_{\ell =1}^{i}\left\{ \left( \tilde{b}_{\ell }%
\left(\tilde{u}^{BL}_{i-\ell ,0}\right)^{\prime }\right) ^{\prime }-\tilde{c}_{\ell }\tilde{u}^{BL}_{i-\ell ,0}\right\} ,i\geq 1 \\
-\tilde{b}_{0}\left(\tilde{u}^{BL}_{0,1}\right)^{\prime \prime }+\tilde{c}_{0}\tilde{u}^{BL}_{0,1}=0 \\
-\tilde{b}_{0}\left(\tilde{u}^{BL}_{i,1}\right)^{\prime \prime }+\tilde{c}_{0}\tilde{u}^{BL}_{i,1}=%
\sum_{\ell =1}^{i}\left\{ \left( \tilde{b}_{\ell }\left(\tilde{u}^{BL}_{i-\ell ,1}\right)^{\prime }\right) ^{\prime }-\tilde{c}_{\ell }\tilde{u}^{BL}_{i-\ell ,1}\right\} ,i\geq 1 \\
-\tilde{b}_{0}\left(\tilde{u}^{BL}_{0,j}\right)^{\prime \prime }+\tilde{c}_{0}\tilde{u}^{BL}_{0,j}=-\left(\tilde{u}^{BL}_{0,j-2}\right)^{(4)},j\geq 2 \\
-\tilde{b}_{0}\left(\tilde{u}^{BL}_{i,j}\right)^{\prime \prime }+\tilde{c}_{0}\tilde{u}^{BL}_{i,j}=-\left(\tilde{u}^{BL}_{i,j-2}\right)^{(4)}+\sum_{\ell =1}^{i}\left\{ \left( \tilde{b}_{\ell }%
\left(\tilde{u}^{BL}_{i-\ell ,j}\right)^{\prime }\right) ^{\prime }-\tilde{c}_{\ell }\tilde{u}^{BL}_{i-\ell ,j}\right\} ,i\geq 1,j\geq 2%
\end{array}%
\right\}.  \label{ut_ij}
\end{equation}
Correspondingly, for the boundary layer $\bar{u}_{i,j}^{BL}$, i.e. for the (fast)
variable $\bar{x}=x\varepsilon _{2}/\varepsilon _{1}$, we first write 
$b(x)=b(\varepsilon _{1}\varepsilon _{2}x/\varepsilon
_{1}\varepsilon _{2})=b(\varepsilon _{1}\bar{x}/\varepsilon _{2})$ and use a Taylor
series to get
\begin{equation*}
b(\bar{x})=\sum_{\ell =0}^{\infty } \left( \frac{\varepsilon _{1}}{\varepsilon _{2}}\right) ^{\ell }\bar{b}_{\ell }(\bar{x}),
\end{equation*}
where $\bar{b}_{\ell }(\bar{x})=$ $\frac{b^{(\ell )}(0)}{\ell !}\bar{x}^{\ell }$, and
similarly for $c$. Then, we find that (see the Appendix for the details)
\begin{equation}
\left. 
\begin{array}{c}
\left(\bar{u}^{BL}_{i,0}\right)^{(4)}-\bar{b}_{0}\left(\bar{u}^{BL}_{i,0}\right)^{\prime \prime }=0,i\geq 0 \\
\left(\bar{u}^{BL}_{0,1}\right)^{(4)}-\bar{b}_{0}\left(\bar{u}^{BL}_{0,1}\right)^{\prime \prime }=0 \\
\left(\bar{u}^{BL}_{i,1}\right)^{(4)}-\bar{b}_{0}\left(\bar{u}^{BL}_{i,1}\right)^{\prime \prime }=\left( \bar{b}_{1}\left(\bar{u}^{BL}_{i-1,0}\right)^{\prime }\right) ^{\prime },i\geq 1 \\
\left(\bar{u}^{BL}_{0,j}\right)^{(4)}-\bar{b}_{0}\left(\bar{u}^{BL}_{0,j}\right)^{\prime \prime }=-\bar{c}_{0}\bar{u}^{BL}_{0,j-2},j\geq 2 \\
\left(\bar{u}^{BL}_{1,2}\right)^{(4)}-\bar{b}_{0}\left(\bar{u}^{BL}_{1,2}\right)^{\prime \prime }=-\bar{c}_{0}\bar{u}^{BL}_{1,0}+\left( \bar{b}_{1}\left(\bar{u}^{BL}_{0,1}\right)^{\prime }\right) ^{\prime } \\
\left(\bar{u}^{BL}_{i,2}\right)^{(4)}-\bar{b}_{0}\left(\bar{u}^{BL}_{i,2}\right)^{\prime \prime }=-\bar{c}_{0}\bar{u}^{BL}_{i,0}+\left( \bar{b}_{2}\left(\bar{u}^{BL}_{i-2,0}\right)^{\prime }\right) ^{\prime }+\left( \bar{b}_{1}\left(\bar{u}^{BL}_{i-1,1}\right)^{\prime }\right) ^{\prime },i\geq 2 \\
\left(\bar{u}^{BL}_{i,j}\right)^{(4)}-\bar{b}_{0}\left(\bar{u}^{BL}_{i,j}\right)^{\prime \prime }=-\bar{c}_{0}\bar{u}^{BL}_{i,j-2}+\left( \bar{b}_{j}\left(\bar{u}^{BL}_{i-j,0}\right)^{\prime }\right) ^{\prime }+\left( \bar{b}_{j-1}\left(\bar{u}^{BL}_{i-j+1,1}\right)^{\prime }\right) ^{\prime }+ \\
+\sum_{\ell =1}^{j-2}\left\{ \left( \bar{b}_{\ell}\left(\bar{u}^{BL}_{i-\ell ,j-\ell}\right)^{\prime }\right) ^{\prime }-\bar{c}_{\ell}\bar{u}^{BL}_{i-\ell ,j-\ell -2}\right\} ,i\geq 3,j=3,\ldots ,i \\
\left(\bar{u}^{BL}_{i,j}\right)^{(4)}-\bar{b}_{0}\left(\bar{u}^{BL}_{i,j}\right)^{\prime \prime }=-\bar{c}_{0}\bar{u}^{BL}_{i,j-2}+\left( \bar{b}_{i}\left(\bar{u}^{BL}_{0,1}\right)^{\prime }\right) ^{\prime }+ \\
+\sum_{\ell =1}^{i-1}\left\{ \left( \bar{b}_{\ell}\left(\bar{u}^{BL}_{i-\ell ,j-\ell}\right)^{\prime }\right) ^{\prime }-\bar{c}_{\ell}\bar{u}^{BL}_{i-\ell ,j-\ell -2}\right\} ,i\geq 2,j=i+1 \\
\left(\bar{u}^{BL}_{i,j}\right)^{(4)}-\bar{b}_{0}\left(\bar{u}^{BL}_{i,j}\right)^{\prime \prime }=-\bar{c}_{0}\bar{u}^{BL}_{i,j-2}+ \\
+\sum_{\ell =1}^{i}\left\{ \left( \bar{b}_{\ell}\left(\bar{u}^{BL}_{i-\ell ,j-\ell}\right)^{\prime }\right) ^{\prime }-\bar{c}_{\ell}\bar{u}^{BL}_{i-\ell ,j-\ell -2}\right\} ,i\geq 1,j>i+1
\end{array}%
\right\}.  \label{ubar_ij}
\end{equation}
Both (\ref{ut_ij}) and (\ref{ubar_ij}) are supplemented with the boundary
conditions$\;\forall \;i,j:$
\begin{equation}
\left. 
\begin{array}{c}
\tilde{u}^{BL}_{i,j}(0)+\bar{u}^{BL}_{i,j}(0)=-u_{i,j}(0) \\ 
\left(\tilde{u}^{BL}_{i,j}\right)^{\prime }(0)+\left(\bar{u}^{BL}_{i,j}\right)^{\prime }(0)=-u_{i,j}^{\prime }(0)
\\ 
\lim_{x\rightarrow \infty }\tilde{u}^{BL}_{i,j}(x)=0 \; , \; \lim_{x\rightarrow \infty }\bar{u}^{BL}_{i,j}(x)=0 \\
\lim_{x\rightarrow \infty }\left(\tilde{u}^{BL}_{i,j}\right)^{\prime }(x)=0 \; , \; \lim_{x\rightarrow \infty }\left(\bar{u}^{BL}_{i,j}\right)^{\prime }(x)=0
\end{array}
\right\} .  \label{bcBL}
\end{equation}
Analogous problems are satisfied by $\check{u}^{BL}_{i,j}$ and $\hat{u}^{BL}_{i,j}$. A remark is in order regarding the boundary conditions involving limits: they are meant to enforce \emph{decay} in the functions, even though strictly speaking, they are not defined on $(0, \infty)$.

\section{Derivative estimates\label{erranal}}

In this section we present results pertaining to the regularity of each term
in the decomposition for $u$. In particular, we give derivative bounds that
are explicit in the differentiation order and in the singular perturbation
parameters. We begin with the smooth part.

\begin{lemma}\label{lem:complex}
Let $u_{i,j}$ be defined by (\ref{uij}) and assume  (\ref{analytic})--(\ref{eps_relation}) hold. Then there exist positive constants $C,K$ and a complex neighbourhood $G$ of $\overline{I}=[0,1]$, such that the complex extension of 
$u_{i,j}$ (denoted again by $u_{i,j}$) satisfies for $\delta \in (0,1)$, 
\begin{equation}
\left|u_{i, j}(z)\right| \leq C \delta^{-i} K^i i^i \quad \forall z \in G_\delta=\{z \in G: \operatorname{dist}(z, \partial G)>\delta\} .
\end{equation}
\end{lemma}

\begin{proof}
The proof is by induction on $i$. The case $i=0$ holds trivially, so assume the result holds for $i$ and establish it for $i+1$. 
Let $\tau \in(0,1), x \in G_\delta$ and $K>0$ be a constant so that $\left(\frac{24}{K^4}+\frac{3}{K^2}\right) \leq C$. 
Then by (\ref{uij}), the induction hypothesis with $G_{(1-\tau) \delta} \supset G_\delta$, and Cauchy's Integral Theorem  for derivatives 
applied to $u_{i, j}$, where the integration path is chosen as the circle of radius $\tau \delta$ about $x$, we obtain
$$
\begin{aligned}
\left|u_{i+1, j}(z)\right| \leq & C\left(\left|u_{i-3, j-2}^{(4)}(z)\right|+\left|u_{i-1, j}^{\prime \prime}(z)\right|+\left|u_{i-1, j}^{\prime}(z)\right|\right) \\
\leq & C\left(\frac{24}{(\tau \delta)^4}((1-\tau) \delta)^{-i+3} K^{i-3}(i-3)^{i-3}+\frac{2}{(\tau \delta)^2}((1-\tau) \delta)^{-i+1}\cdot \right. \\
& \left.\cdot K^{i-1}(i-1)^{i-1}+\frac{1}{(\tau \delta)}((1-\tau) \delta)^{-i+1} K^{i-1}(i-1)^{i-1}\right) \\
\leq & C \delta^{-i-1} K^{i+1}(i+1)^{i+1}\left(\frac{1}{K^4} \frac{1}{(i+1)^4} \frac{24}{\tau^4(1-\tau)^{i-3}}\left(\frac{i-3}{i+1}\right)^{i-3}+\right. \\
& \left.+\frac{1}{K^2} \frac{1}{(i+1)^2} \frac{3}{\tau^2(1-\tau)^{i-1}}\left(\frac{i-1}{i+1}\right)^{i-1}\right) .
\end{aligned}
$$
Choose $\tau=1 /(i+1)$. Then
$$
\left|u_{i+1, j}(z)\right| \leq C \delta^{-i-1} K^{i+1}(i+1)^{i+1}\left(\frac{24}{K^4}+\frac{3}{K^2}\right),
$$
so by the choice of $K$ the expression in brackets is bounded, and this completes the proof.
\end{proof}

Using the above, and Cauchy's Integral Theorem for Derivatives, we immediately get (cf. \cite{melenk})

\begin{lemma}\label{lem:u_n_ij}
Let $u_{i,j}$ be defined by (\ref{uij}) and assume  (\ref{analytic})--(\ref{eps_relation}) hold. Then there exist
positive constants $C, K_1, K_2,$ such that
\begin{equation}
\left\|u_{i, j}^{(n)}\right\|_{\infty, I} \leq C n ! K_1^n i^i K_2^i \quad \forall n \in \mathbb{N} .
\end{equation}
\end{lemma}

The above result shows that the smooth part of the solution is analytic if the data are analytic. To deal with the boundary layers, we will need the auxiliary result given below.

\begin{lemma}
\label{lemma_aux0} Let $\lambda_1 \in \mathbb{C}$ with $Re(\lambda_1)>0,Re(\lambda_1^{2})>0$. Let $F_1$ be an entire function satisfy- ing, for some $%
C_{F_1}>0,$ $j\in \mathbb{N}_{0},$ $q_1\geq (j+1/2)/|\lambda_1|>0,$ 
\begin{equation*}
\left\vert F_1(z)\right\vert \leq C_{F_1}e^{-Re(\lambda_1 z)}\left( q_1+\left\vert
z\right\vert \right) ^{j} \; \; \forall \; z\in \mathbb{C}.
\end{equation*}%
Similarly, let $j\in \mathbb{N}, \kappa, \lambda_2 \in \mathbb{C}^{+},$ with $\lambda_2 = \kappa^2$ and let $F_2$ be an entire function satisfying, for some $C_{F_2}>0,$ $q_2\geq \frac{4j}{|\kappa|}>0,$
\begin{equation*}
\left\vert F_2(z)\right\vert \leq C_{F_2}e^{-Re(\kappa z)}\left( q_2+\left\vert
z\right\vert \right) ^{2j-1} \; \; \forall \; z\in \mathbb{C}.
\end{equation*}%
With $\alpha_1, \alpha_2 \in \mathbb{C}$, let $v,w:(0,\infty )\rightarrow \mathbb{C}$ be the solutions of the problems
\begin{equation*}
-v^{\prime \prime }+\lambda_1^{2}v=F_1\text{ on }(0,\infty
) \; , \; v(0)=\alpha_1 - w(0) \; , \; \lim_{x\rightarrow \infty }v(x)=0,
\end{equation*}%
\begin{equation*}
w^{(4)}-\lambda_2 w^{(2)}=F_2\text{ on }(0,\infty ) \; , \; w^{\prime }(0)=\alpha_2 - v^{\prime }(0) \; , \; \lim_{x\rightarrow \infty}w(x)=0.
\end{equation*}%
Then $v,w$ can be extended to entire functions (denoted again by $v,w$) which satisfy
\begin{equation*}
\left\vert v(z)\right\vert \leq C \left[ \frac{C_{F}}{j+1}\left(
q_*+\left\vert z\right\vert \right) ^{j+1}+\left\vert \alpha_1 \right\vert+\left\vert \alpha_2 \right\vert \right] e^{-Re(\lambda_1 z)} \; \; \forall \; z\in \mathbb{C},
\end{equation*}%
\begin{equation*}
\left\vert w(z)\right\vert \leq C \left[ \frac{C_{F}}{2j}\left( q_*+\left\vert z\right\vert \right) ^{2j}+\left\vert \alpha_1 \right\vert+\left\vert \alpha_2 \right\vert \right] e^{-Re(\kappa z)} \; \; \forall \; z\in \mathbb{C},
\end{equation*}%
where $C>0,q_* = \max \{q_1,q_2^2 \},C_F = \max \{C_{F_1},C_{F_2} \}$.
\end{lemma}
\begin{proof}
This is a combination of \cite[Lemma 7.3.6]{melenk} and \cite[Proposition 2.4.10]{PC} (see also \cite[Proposition 3]{PCarXiv}). The only difference is that in our case we have different boundary conditions involving $v(0),v'(0),w(0),w'(0)$. Inspecting the proofs of the aforementioned results, we see that through a Green's function representation, we obtain 
\[
v(z) = G_v(z) + (\alpha_1 - w(0)) e^{-\lambda_1 z} \; , \; w(z) = G_w(z) - \frac{1}{\kappa} (\alpha_2 - v'(0)) e^{-\kappa z},
\]
\[
v'(z) = G'_v(z) -\lambda_1 (\alpha_1 - w(0)) e^{-\lambda_1 z} \; , \; w'(z) = G'_w(z) + (\alpha_2 - v'(0)) e^{-\kappa z},
\]
where
\[
|G_v(z)| \leq \frac{C_{F_1}}{|\lambda_1 |}\left(
q_1+\left\vert z\right\vert \right) ^{j+1}\left( j+1\right) ^{-1} e^{-Re(\lambda_1 z)},
\]
\[
|G_w(z)| \leq \frac{C_{F_2}}{2j}\left( q_2+\left\vert z\right\vert \right) ^{2j} e^{-Re(\kappa z)}
\]
and
\[
|G'_v(z)| \leq C_v |G_v(z)| \; , \; |G'_w(z)| \leq C_w |G_w(z)|, 
\]
with the constants $C_v,C_w$ independent of $j$, but depending on $\lambda_1, \kappa, q_1, q_2, C_{F_1}, C_{F_2}$.
Combining the above, we get
\[
v(0) = G_v(0) + (\alpha_1-w(0)) \; , \; w(0) = G_w(0) - \frac{1}{\kappa} (\alpha_2-v'(0)),
\]
\[
v'(0) = G'_v(0) -\lambda_1 (\alpha_1-w(0)) \; , \; w'(0) = G'_w(0) + (\alpha_2-v'(0)).
\]
The above $4 \times 4$ system of equations, in the unknowns $v(0), w(0), v'(0), w'(0)$, once solved yields
\begin{eqnarray*}
v(0) &=& \frac{(\kappa - \lambda_1)G_v(0) - G'_v(0) - \kappa G_w(0) + \kappa \alpha_1 + \alpha_2}{\kappa - \lambda_1},\\
w(0) &=& \frac{\kappa G_w(0) + G'_v(0) - \lambda_1 \alpha_1 - \alpha_2}{\kappa - \lambda_1},\\
v'(0) &=& \frac{\kappa \lambda_1 G_w(0) + \kappa G'_v(0) - \kappa \lambda_1 \alpha_1 - \lambda_1 \alpha_2}{\kappa - \lambda_1},\\
w'(0) &=& \frac{(\kappa - \lambda_1)G'_w(0) - \kappa \lambda_1 G_w(0) - \kappa G'_v(0) + \kappa \lambda_1 \alpha_1 + \kappa \alpha_2}{\kappa - \lambda_1}.
\end{eqnarray*}
An application of \cite[Lemma 7.3.6]{melenk} gives the bound
\begin{eqnarray*}
|v(z)| &\leq& e^{-Re(\lambda_1 z)} \left[ \frac{C_{F_1}}{|\lambda_1|} \left(
q_1 + \left\vert z\right\vert \right) ^{j+1} \left( j+1\right) ^{-1} + \right. \\
&+& \left. \frac{\left\vert (\kappa - \lambda_1) G_v(0) \right\vert + \left\vert G'_v(0) \right\vert + \left\vert \kappa G_w(0) \right\vert + |\kappa \alpha_1| + |\alpha_2|}{|\kappa - \lambda_1|} \large\right] \\
&\leq& C e^{-Re(\lambda_1 z)} \left[ \frac{C_{F_1}}{|\lambda_1|} \left(
q_1 + \left\vert z\right\vert \right) ^{j+1} \left( j+1\right) ^{-1} + \frac{C_{F_1}}{|\lambda_1|} (q_1)^{j+1} \left( j+1\right) ^{-1} + \right. \\
&+& \left. \frac{C_{F_2}}{2j} (q_2)^{2j} + |\alpha_1| + |\alpha_2| \right].
\end{eqnarray*}
The desired result for $v$ follows. 

Similarly, to treat $w$ we use \cite[Proposition 2.4.10]{PC} to get
\begin{equation*}
|w(z)| \leq C \left[ \frac{C_{F_2}}{2j} (q_2 + |z|)^{2j} + \frac{\left\vert G'_w(0) \right\vert + |\alpha_2| + \left\vert v'(0) \right\vert}{|\kappa|} \right] e^{-Re(\kappa z)}.
\end{equation*}
We bound the term $|v'(0)|$ as follows:
\begin{eqnarray*}
|v'(0)| &\leq& \frac{\left\vert \kappa \lambda_1 G_w(0) \right\vert + \left\vert \kappa G'_v(0) \right\vert + |\kappa \lambda_1 \alpha_1| + |\lambda_1 \alpha_2|}{|\kappa - \lambda_1|}\\
&\leq& C \left\{ \frac{C_{F_1}}{|\lambda_1|} (q_1)^{j+1} \left( j+1\right) ^{-1} + \frac{C_{F_2}}{2j} (q_2)^{2j} + |\alpha_1| + |\alpha_2| \right\}.
\end{eqnarray*}
Thus,
\[
|w(z)| \leq C \left[ \frac{C_{F}}{2j} (q_* + |z|)^{2j} + |\alpha_1| + |\alpha_2| \right] e^{-Re(\kappa z)}
\]
as desired.
\end{proof}

We are now in a position to establish the following.

\begin{lemma}
\label{lem:ut_ij}
The functions $\tilde{u}^{BL}_{i,j}, \bar{u}^{BL}_{i,j}$ which satisfy (\ref{ut_ij})--(\ref{bcBL}), are entire and there exist
positive constants $C, \tilde{\gamma}, \bar{\gamma},$ such that
\begin{eqnarray*}
\vert \tilde{u}^{BL}_{i,j}(z) \vert &\leq& C \tilde{\gamma}^{i+j} \frac{(2i+j+|z|)^{2i+j}}{i!} e^{-\tilde{\beta} Re(z)} \; \forall \; z \in \mathbb{C}, \; Re(z) > 0,\\
\vert \bar{u}^{BL}_{i,j}(z) \vert &\leq& C \bar{\gamma}^{i+j} \frac{(i+2j+|z|)^{i+2j}}{j!} e^{-\bar{\beta} Re(z)} \; \forall \; z \in \mathbb{C}, \; Re(z) > 0,
\end{eqnarray*}
where $\tilde{\beta}=\sqrt{\tilde{c}_{0} / \tilde{b}_{0}}$ and $\bar{\beta}=\sqrt{\bar{b}_{0}}$.
\end{lemma}
\begin{proof}
We recall that
\begin{equation*}
\tilde{b}_{\ell}(\tilde{x}) = \frac{\tilde{x}^{\ell} b^{(\ell)}(0)}{\ell !} \; , \; \tilde{c}_{\ell}(\tilde{x}) = \frac{\tilde{x}^{\ell} c^{(\ell)}(0)}{\ell !} \; , \; \bar{b}_{\ell}(\bar{x}) = \frac{\bar{x}^{\ell} b^{(\ell)}(0)}{\ell !} \; , \; \bar{c}_{\ell}(\bar{x}) = \frac{\bar{x}^{\ell} c^{(\ell)}(0)}{\ell !}.
\end{equation*}
Consequently, there exist positive constants $C_{\tilde{b}}, \gamma _{\tilde{b}}, C_{\tilde{c}}, \gamma _{\tilde{c}}, C_{\bar{b}}, \gamma _{\bar{b}}, C_{\bar{c}}, \gamma _{\bar{c}},$ depending only on $b,c,$ such that
\begin{equation*}
|\tilde{b}_{\ell}(z)| \leq C_{\tilde{b}} \gamma _{\tilde{b}} |z|^{\ell} \; , \; |\tilde{c}_{\ell}(z)| \leq C_{\tilde{c}} \gamma _{\tilde{c}} |z|^{\ell} \; , \; |\bar{b}_{\ell}(z)| \leq C_{\bar{b}} \gamma _{\bar{b}} |z|^{\ell} \; , \; |\bar{c}_{\ell}(z)| \leq C_{\bar{c}} \gamma _{\bar{c}} |z|^{\ell}.
\end{equation*}
Then, with $K_{2}$ the constant from Lemma \ref{lem:u_n_ij}, we choose $\tilde{\gamma } > \max \{ K_{2}, \gamma _{\tilde{b}}, \gamma _{\tilde{c}} \}$,$\bar{\gamma } > \max \{ K_{2}, \gamma _{\bar{b}}, \gamma _{\bar{c}} \}$ so that
\begin{equation*}
\left[ \frac{\gamma _{\tilde{b}} / \tilde{\gamma }}{1 - \gamma _{\tilde{b}} / \tilde{\gamma }} + \frac{\gamma _{\tilde{c}} / \tilde{\gamma }}{1 - \gamma _{\tilde{c}} / \tilde{\gamma }} \right] < 1
\end{equation*}
and
\begin{equation*}
\left[ \frac{\gamma _{\bar{b}} / \bar{\gamma }}{1 - \gamma _{\bar{b}} / \bar{\gamma }} + \frac{\gamma _{\bar{c}} / \bar{\gamma }}{1 - \gamma _{\bar{c}} / \bar{\gamma }} \right] < 1.
\end{equation*}
From (\ref{ut_ij}) and (\ref{ubar_ij}), we have that $\tilde{u}^{BL}_{0,0}(\tilde{x}),\bar{u}^{BL}_{0,0}(\bar{x})$ are given by 
\begin{equation*}
\tilde{u}^{BL}_{0,0}(\tilde{x})=\tilde{C}_{0,0}e^{-\sqrt{\frac{\tilde{c}_{0}}{%
\tilde{b}_{0}}}\tilde{x}}\; , \; \bar{u}^{BL}_{0,0}(\bar{x})=\bar{C}_{0,0}e^{-\sqrt{\bar{%
b}_{0}}\bar{x}}
\end{equation*}
and from (\ref{bcBL}) we find
\[
\tilde{C}_{0,0} = \frac{u_{0,0}(0) \sqrt{\bar{b}_0}+u'_{0,0}(0)}{\sqrt{\frac{\tilde{c}_0}{\tilde{b}_0}}-\sqrt{\bar{b}_0}} \; , \;
\bar{C}_{0,0} = \frac{-u_{0,0}(0) \sqrt{\frac{\tilde{c}_0}{\tilde{b}_0}}-u'_{0,0}(0)}{\sqrt{\frac{\tilde{c}_0}{\tilde{b}_0}}-\sqrt{\bar{b}_0}}.
\]
A similar result holds for $\tilde{u}^{BL}_{0,1}(\tilde{x})$ and $\bar{u}^{BL}_{0,1}(\bar{x})$, involving 
$\tilde{C}_{0,1}, \bar{C}_{0,1}$ given in terms of $u_{0,1}(0)$ and $u'_{0,1}(0)$, hence the desired estimates 
hold for $i=0,j=0,1$. Let us consider $i=0,j=2$: 
\[
-\tilde{b}_0 \left(\tilde{u}^{BL}_{0,2}\right)^{\prime \prime} + \tilde{c}_0 \tilde{u}^{BL}_{0,2} = - \left(\tilde{u}^{BL}_{0,0}\right)^{(4)}
\]
\[
\left( \bar{u}^{BL}_{0,2}\right)^{(4)} - \bar{b}_0 \left(\bar{u}^{BL}_{0,2} \right)^{\prime \prime}= - \bar{c}_0 \bar{u}^{BL}_{0,0}.
\]
Solving the above equations, we find
\[
\tilde{u}^{BL}_{0,2}(\tilde{x}) = \tilde{C}_{0,2} e^{-\sqrt{\frac{\tilde{c}_0}{\tilde{b}_0}} \tilde{x}} - A \left(\tilde{u}^{BL}_{0,0}\right)^{(4)}(\tilde{x})\; , \;
\bar{u}^{BL}_{0,2}(\bar{x}) = \bar{C}_{0,2} e^{-\sqrt{\bar{b}_0} \bar{x}} - B \bar{c}_0 \bar{u}^{BL}_{0,0}(\bar{x}),
\]
with $A, B \in \mathbb{R}$ depending only on $\tilde{b}_0, \tilde{c}_0, \tilde{C}_{0,0}, \bar{b}_0, \bar{c}_0, \bar{C}_{0,0}$. Since 
\[
\left(\tilde{u}^{BL}_{0,2}\right)^{(k)}(0) + \left(\bar{u}^{BL}_{0,2}\right)^{(k)}(0) = - u^{(k)}_{0,2}(0) \; , \; k=0,1,
\]
we can calculate the constants $\tilde{C}_{0,2}, \bar{C}_{0,2}$, in terms of $\tilde{b}_0, \tilde{c}_0, \bar{b}_0, u_{0,2}(0), u'_{0,2}(0)$. This shows that the desired results hold for $i=0,j=2$.

We proceed with 
induction on $j>2$, while keeping $i$ fixed at 0. Assuming the result holds for $j$, we will establish it
for $j+1$. The function $\tilde{u}^{BL}_{0,j+1}$ satisfies (see eq. (\ref{ut_ij}))
\[
-\tilde{b}_{0}\left(\tilde{u}^{BL}_{0,j+1}\right)^{\prime \prime }+\tilde{c}_{0}\tilde{u}^{BL}_{0,j+1}=-%
\left(\tilde{u}^{BL}_{0,j-1}\right)^{(4)} \; , \; j\geq 2.
\]
From the induction hypothesis and Cauchy's Integral Theorem for Derivatives, we get
\[
\vert \left(\tilde{u}^{BL}_{0,j-1}\right)^{(4)}(z) \vert \leq C \tilde{\gamma}^{j-1} (j-1+|z|)^{j-1} e^{-\tilde{\beta} Re(z)}
\]
and by Lemma \ref{lemma_aux0},
\begin{eqnarray*}
\vert \tilde{u}^{BL}_{0,j+1} (z) \vert &\leq& C e^{-\tilde{\beta} Re(z)} \left[ \tilde{\gamma}^{j-1} \frac{(j-1+|z|)^{j}}{j} +
|u_{0,j+1}(0)| + |u'_{0,j+1}(0)| \right]  \notag \\
&\leq& C e^{-\tilde{\beta} Re(z)} \tilde{\gamma}^{j+1} (j+1+|z|)^{j+1}\left[ \frac{1}{\tilde{\gamma}^2 j (j+1+|z|)} + \frac{\tilde{C}}{\tilde{\gamma}^{j+1} (j+1+|z|)^{j+1}} \right] \\
&\leq& C e^{-\tilde{\beta} Re(z)} \tilde{\gamma}^{j+1} (j+1+|z|)^{j+1},
\end{eqnarray*}
which gives the result for $\tilde{u}_{0,j+1}$.

We next show the result for $\bar{u}^{BL}_{0,j+1}$, which satisfies (see eq. (\ref{ubar_ij}))
\[
\left(\bar{u}^{BL}_{0,j+1}\right)^{(4)}-\bar{b}_{0}\left(\bar{u}^{BL}_{0,j+1}\right)^{\prime \prime }=-\bar{c}_{0}\bar{u}^{BL}_{0,j-1} \; , \; j\geq 2.
\]
By the induction hypothesis, we have
\[
\vert \bar{u}^{BL}_{0,j-1} (z) \vert \leq C \bar{\gamma }^{j-1} \frac{(2(j-1)+|z|)^{2(j-1)}}{(j-1)!} e^{-\bar{\beta} Re(z)}
\]
and using Lemma \ref{lemma_aux0}, we obtain
\begin{eqnarray*}
\vert \bar{u}^{BL}_{0,j+1} (z) \vert &\leq& C e^{-\bar{\beta} Re(z)} \left[ \bar{\gamma}^{j-1} \frac{(2(j-1)+|z|)^{2j-1}}{(2j-1) (j-1)!} +
|u_{0,j+1}(0)| + |u'_{0,j+1}(0)| \right]  \notag \\
&\leq& C e^{-\bar{\beta} Re(z)} \bar{\gamma}^{j+1} \frac{(2(j+1)+|z|)^{2(j+1)}}{(j+1)!} \left[ \frac{j(j+1)}{\bar{\gamma}^2 (2j-1) (2(j+1)+|z|)^3} + \right. \\
&+& \left. \frac{\bar{C} (j+1)^{j+1}}{\bar{\gamma}^{j+1} (2(j+1)+|z|)^{2(j+1)}} \right] \\
&\leq& C e^{-\bar{\beta} Re(z)} \bar{\gamma}^{j+1} \frac{(2(j+1)+|z|)^{2(j+1)}}{(j+1)!}.
\end{eqnarray*}

We finally consider the case $i\geq 1,j\geq 0$. We first assume that the result holds for $\tilde{u}^{BL}_{i,j}, j\geq 2$ and we will establish it for $\tilde{u}^{BL}_{i,j+1}$, with $i\geq 1$ fixed (but arbitrary). By (\ref{ut_ij}), the function $\tilde{u}^{BL}_{i,j+1}$ satisfies
$$
-\tilde{b}_0\left(\tilde{u}_{i, j+1}^{B L}\right)^{\prime \prime}+\tilde{c}_0 \tilde{u}_{i, j+1}^{B L}=-\left(\tilde{u}_{i, j-1}^{B L}\right)^{(4)}+\sum_{\ell=1}^i\left\{\tilde{b}_{\ell}\left(\tilde{u}_{i-\ell, j+1}^{B L}\right)^{\prime \prime}+\tilde{b}_{\ell}^{\prime}\left(\tilde{u}_{i-\ell, j+1}^{B L}\right)^{\prime}-\tilde{c}_{\ell} \tilde{u}_{i-\ell, j+1}^{B L}\right\} .
$$
In order to use Lemma \ref{lemma_aux0}, we need to bound the right hand side above:
$$
\tilde{G}(z):=-\left(\tilde{u}_{i, j-1}^{B L}\right)^{(4)}+\sum_{\ell=1}^i\left\{\tilde{b}_{\ell}\left(\tilde{u}_{i-\ell, j+1}^{B L}\right)^{\prime \prime}+\tilde{b}_{\ell}^{\prime}\left(\tilde{u}_{i-\ell, j+1}^{B L}\right)^{\prime}-\tilde{c}_{\ell} \tilde{u}_{i-\ell, j+1}^{B L}\right\} .
$$
By the induction hypothesis and Cauchy's Integral Theorem for Derivatives, we find
$$
\begin{aligned}
|\tilde{G}(z)| \leq & \left|\left(\tilde{u}_{i, j-1}^{B L}\right)^{(4)}\right|+C \sum_{\ell=1}^i\left\{\left|\tilde{b}_{\ell}\right|\left|\left(\tilde{u}_{i-\ell, j+1}^{B L}\right)^{\prime \prime}\right|+\left|\tilde{b}_{\ell}^{\prime}\right|\left|\left(\tilde{u}_{i-\ell, j+1}^{B L}\right)^{\prime}\right|+\left|\tilde{c}_{\ell}\right|\left|\tilde{u}_{i-\ell, j+1}^{B L}\right|\right\} \\
\leq & C \tilde{\gamma}^{i+j-1} \frac{(2 i+j-1+|z|)^{2 i+j-1}}{i!} e^{-\tilde{\beta} R e(z)}+C \sum_{\ell=1}^i|z|^{\ell}\left(\gamma_{\tilde{b}}^{\ell}+\gamma_{\tilde{c}}^{\ell}\right)\left|\tilde{u}_{i-\ell, j+1}^{B L}\right| \\
\leq & C \tilde{\gamma}^{i+j-1} \frac{(2 i+j-1+|z|)^{2 i+j-1}}{i!} e^{-\tilde{\beta} R e(z)}+ \\
& +C \sum_{\ell=1}^i|z|^{\ell}\left(\gamma_{\bar{b}}^{\ell}+\gamma_{\tilde{c}}^{\ell}\right) \tilde{\gamma}^{i+j+1-\ell} \frac{(2(i-\ell)+j+1+|z|)^{2(i-\ell)+j+1}}{(i-\ell)!} e^{-\tilde{\beta} R e(z)} \\
\leq & C \tilde{\gamma}^{i+j+1} \frac{(2 i+j+1+|z|)^{2 i+j}}{(i-1)!} e^{-\tilde{\beta} R e(z)}\left\{\frac{1}{\tilde{\gamma}^2 i(2 i+j+1+|z|)}+\sum_{\ell=1}^{\infty}\left[\frac{\gamma_{\tilde{b}}^{\ell}}{\tilde{\gamma}^{\ell}}+\frac{\gamma_{\tilde{c}}^{\ell}}{\tilde{\gamma}^{\ell}}\right]\right\} .
\end{aligned}
$$
The assumptions on $\tilde{\gamma}$ allow us to infer that the geometric series converges to a bounded quantity, hence the expression in braces is also bounded and we have the necessary setup to use Lemma \ref{lemma_aux0}, which yields
\begin{eqnarray*}
\vert \tilde{u}_{i,j+1}(z) \vert &\leq& C e^{-\tilde{\beta} Re(z)} \left[ \tilde{\gamma}^{i+j+1} \frac{(2i+j+1+|z|)^{2i+j+1}}{(2i+j+1) (i-1)!}
+|{u}_{i,j+1}(0)|+|{u}'_{i,j+1}(0)| \right] \\
&\leq& C e^{-\tilde{\beta} Re(z)} \tilde{\gamma}^{i+j+1} \frac{(2i+j+1+|z|)^{2i+j+1}}{i!} \left[ \frac{i}{(2i+j+1)}+\frac{\tilde{C} i^{2i}}{\tilde{\gamma}^{j+1} (2i+j+1+|z|)^{2i+j+1}} \right] \\
&\leq& C e^{-\tilde{\beta} Re(z)} \tilde{\gamma}^{i+j+1} \frac{(2i+j+1+|z|)^{2i+j+1}}{i!}.
\end{eqnarray*}
Following the same steps as above, we can show the desired results for $\tilde{u}^{BL}_{i,0},\tilde{u}^{BL}_{i,1}, i\geq 1$. Now, we assume that the result holds for $\bar{u}^{BL}_{i,j}, j > i+1$ and we will establish it for $\bar{u}^{BL}_{i,j+1}$, with $i\geq 1$ fixed. By (\ref{ubar_ij}), the function $\bar{u}^{BL}_{i,j+1}$ satisfies
$$
\left(\bar{u}_{i, j+1}^{B L}\right)^{(4)}-\bar{b}_0\left(\bar{u}_{i, j+1}^{B L}\right)^{\prime \prime}=-\bar{c}_0 \bar{u}_{i, j-1}^{B L}+\sum_{\ell=1}^i\left\{\bar{b}_{\ell}\left(\bar{u}_{i-\ell, j+1-\ell}^{B L}\right)^{\prime \prime}+\bar{b}_{\ell}^{\prime}\left(\bar{u}_{i-\ell, j+1-\ell}^{B L}\right)^{\prime}-\bar{c}_{\ell} \bar{u}_{i-\ell, j-1-\ell}^{B L}\right\} .
$$
We bound the right hand side (which we call $\bar{G}$)  as follows, with the aid of the induction hypothesis and Cauchy's Integral Theorem for Derivatives:
$$
\begin{aligned}
& |\bar{G}(z)| \leq\left|\bar{c}_0\right|\left|\bar{u}_{i, j-1}^{B L}\right|+C \sum_{\ell=1}^i\left\{\left|\bar{b}_{\ell}\right|\left|\left(\bar{u}_{i-\ell, j+1-\ell}^{B L}\right)^{\prime \prime}\right|+\left|\bar{b}_{\ell}^{\prime}\right|\left|\left(\bar{u}_{i-\ell, j+1-\ell}^{B L}\right)^{\prime}\right|+\left|\bar{c}_{\ell}\right|\left|\bar{u}_{i-\ell, j-1-\ell}^{B L}\right|\right\} \\
& \leq C \bar{\gamma}^{i+j-1} \frac{(i+2(j-1)+|z|)^{i+2(j-1)}}{(j-1)!} e^{-\bar{\beta} R e(z)}+ \\
& +C \sum_{\ell=1}^i|z|^{\ell}\left(\gamma_{\bar{\ell}}^{\ell}\left|\bar{u}_{i-\ell, j+1-\ell}^{B L}\right|+\gamma_{\bar{c}}^{\ell}\left|\bar{u}_{i-\ell, j-1-\ell}^{B L}\right|\right) \\
& \leq C \bar{\gamma}^{i+j-1} \frac{(i+2(j-1)+|z|)^{i+2(j-1)}}{(j-1)!} e^{-\bar{\beta} R e(z)}+ \\
& +C \sum_{\ell=1}^i|z|^{\ell} \gamma_{\bar{b}}^{\ell} \bar{\gamma}^{i+j+1-2 \ell} \frac{(i-\ell+2(j+1-\ell)+|z|)^{i-\ell+2(j+1-\ell)}}{(j+1-\ell)!} e^{-\bar{\beta} R e(z)}+ \\
& +C \sum_{\ell=1}^i|z|^{\ell} \gamma_{\bar{c}}^{\ell} \bar{\gamma}^{i+j-1-2 \ell} \frac{(i-\ell+2(j-1-\ell)+|z|)^{i-\ell+2(j-1-\ell)}}{(j-1-\ell)!} e^{-\bar{\beta} R e(z)} \\
& \leq C \bar{\gamma}^{i+j+1} \frac{(i+2(j+1)+|z|)^{i+2 j+1}}{j!} e^{-\bar{\beta} R e(z)}\left\{\frac{j}{\bar{\gamma}^2(i+2(j+1)+|z|)^3}+\right. \\
& \left.+\frac{1}{\bar{\gamma}(i+2(j+1)+|z|)} \sum_{\ell=1}^{\infty} \frac{\gamma_{\bar{b}}^{\ell}}{\bar{\gamma}^{\ell}}+\frac{j(j-1)}{\bar{\gamma}^3(i+2(j+1)+|z|)^5} \sum_{\ell=1}^{\infty} \frac{\gamma_{\bar{c}}^{\ell}}{\bar{\gamma}^{\ell}}\right\} . \\
&
\end{aligned}
$$
The expression in braces is bounded, since both sums are convergent geometric series due to the assumptions on $\bar{\gamma}$, and an application of Lemma \ref{lemma_aux0} gives
$$
\begin{aligned}
\left|\bar{u}_{i, j+1}^{B L}(z)\right| \leq & C e^{-\bar{\beta} R e(z)}\left[\bar{\gamma}^{i+j+1} \frac{(i+2(j+1)+|z|)^{i+2(j+1)}}{(i+2(j+1)) j!}+\left|u_{i, j+1}(0)\right|+\left|u_{i, j+1}^{\prime}(0)\right|\right] \\
\leq & C e^{-\bar{\beta} R e(z)} \bar{\gamma}^{i+j+1} \frac{(i+2(j+1)+|z|)^{i+2(j+1)}}{(j+1)!} \times \\
& \times\left[\frac{j+1}{(i+2(j+1))}+\frac{\bar{C} i^i(j+1)^{j+1}}{\bar{\gamma}^{j+1}(i+2(j+1)+|z|)^{i+2(j+1)}}\right] \\
\leq & C e^{-\bar{\beta} R e(z)} \bar{\gamma}^{i+j+1} \frac{(i+2(j+1)+|z|)^{i+2(j+1)}}{(j+1)!} .
\end{aligned}
$$
Similarly, we can show the result for the remaining functions $\bar{u}_{i,j},i \geq 1$ and the proof is complete.
\end{proof}

Using Cauchy's Integral Theorem for Derivatives, we may show the following.

\begin{lemma}
\label{lem:Dut_ij}
Let the functions $\tilde{u}_{i,j}, \bar{u}_{i,j}$ be defined by (\ref{ut_ij})--(\ref{bcBL}). Then, there exist
positive constants $C, \tilde{K}, \bar{K}, \tilde{\gamma}, \bar{\gamma},$ independent of $\varepsilon_1, \varepsilon_2,$ such that for any $n \in \mathbb{N},$
\begin{eqnarray*}
\vert \left(\tilde{u}^{BL}_{i,j}\right)^{(n)} (z) \vert &\leq& C \tilde{K}^n \tilde{\gamma}^{i+j} \frac{(2i+j)^{2i+j}}{i!} e^{-\tilde{\beta} Re(z)} \; \forall \; z \in \mathbb{C}, \; Re(z) > 0,\\
\vert \left(\bar{u}_{i,j}\right)^{(n)} (z) \vert &\leq& C \bar{K}^n \bar{\gamma}^{i+j} \frac{(i+2j)^{i+2j}}{j!} e^{-\bar{\beta} Re(z)} \; \forall \; z \in \mathbb{C}, \; Re(z) > 0,
\end{eqnarray*}
where $\tilde{\beta}=\sqrt{\tilde{c}_{0} / \tilde{b}_{0}}$ and $\bar{\beta}=\sqrt{\bar{b}_{0}}$.
\end{lemma}
\begin{proof}
This is straight forward, see, e.g. \cite[Corollary 6]{PCarXiv}.
\end{proof}

Completely analogous results hold for the functions $\check{u}^{BL}_{i,j}, \hat{u}^{BL}_{i,j}$.

We now define, for $M \in \mathbb{N}$,
\begin{eqnarray}
u_M(x) &=& \sum_{i=0}^{M} \sum_{j=0}^{M} \varepsilon_2^i (\varepsilon_1/ \varepsilon_2^2)^j u_{i,j}(x), \label{uM} \\
\tilde{u}_{M}^{BL} (\tilde{x}) &=& \sum_{i=0}^{M} \sum_{j=0}^{M} \varepsilon_2^i (\varepsilon_1/ \varepsilon_2^2)^j \tilde{u}^{BL}_{i,j}(\tilde{x}), \label{utM} \allowdisplaybreaks \\
\bar{u}_{M}^{BL} (\bar{x}) &=& \sum_{i=0}^{M} \sum_{j=0}^{M} \varepsilon_2^i (\varepsilon_1/ \varepsilon_2^2)^j \bar{u}^{BL}_{i,j}(\bar{x}), \label{ubM} \\
\check{u}_{M}^{BL} (\check{x}) &=& \sum_{i=0}^{M} \sum_{j=0}^{M} \varepsilon_2^i (\varepsilon_1/ \varepsilon_2^2)^j \check{u}^{BL}_{i,j}(\check{x}), \label{ucM} \\
\hat{u}_{M}^{BL} (\hat{x}) &=& \sum_{i=0}^{M} \sum_{j=0}^{M} \varepsilon_2^i (\varepsilon_1/ \varepsilon_2^2)^j \hat{u}^{BL}_{i,j}(\hat{x}), \label{uhM}
\end{eqnarray}%
\begin{equation}
r_M = u -\left( u_M + \tilde{u}_M^{BL} + \bar{u}_M^{BL} + \check{u}_M^{BL} + \hat{u}_M^{BL}\right) \label{rM}
\end{equation}%
and we have the decomposition
\begin{equation}
u = u_M + \tilde{u}_M^{BL} + \bar{u}_M^{BL} + \check{u}_M^{BL} + \hat{u}_M^{BL} + r_M. \label{decomp}
\end{equation}

We finally prove the main result of the article.

\begin{theorem}\label{thm:main}
Let $u$ be the solution to (\ref{de})--(\ref{bc}) and assume (\ref{analytic})--(\ref{eps_relation}) hold. Then, there
exist positive constants $C,$ $K_1,$ $K_2,$ $\tilde{K},$ $\bar{K},$ $\check{K},$ $\hat{K},$ $\tilde{\gamma},$ $\bar{\gamma},$ $\check{\gamma},$ $\hat{\gamma},$ $\delta ,$ independent of 
$\varepsilon_1, \varepsilon_2,$ such that $\forall \; n \in \mathbb{N}$ and $\forall \; x \in \overline{I},$ there holds
\begin{equation*}
u = u_M + \tilde{u}_M^{BL} + \bar{u}_M^{BL} + \check{u}_M^{BL} + \hat{u}_M^{BL} + r_M,
\end{equation*}
and
\begin{eqnarray*}
\Vert u_M^{(n)} \Vert_{\infty, I} &\leq& C K_1^n n! ,\\
 \vert (\tilde{u}^{BL}_{M})^{(n)}(x) \vert &\leq& C \tilde{K}^n  \frac{1}{\varepsilon _{2}}\left( \frac{%
\varepsilon _{1}}{\varepsilon _{2}}\right) ^{1-n}e^{-\beta \varepsilon
_{2}x/\varepsilon _{1}},\\
 \vert (\bar{u}_{M}^{BL})^{(n)}(x) \vert &\leq& C \bar{K}^n\varepsilon _{2}^{-n}e^{-\beta x/\varepsilon _{2}},\\
\vert (\check{u}_{M}^{BL})^{(n)}(x) \vert &\leq& C \check{K}^n \frac{1}{\varepsilon _{2}}%
\left( \frac{\varepsilon _{1}}{\varepsilon _{2}}\right)
^{1-n}e^{-\beta \varepsilon _{2}(1-x)/\varepsilon _{1}} , \\
\vert (\hat{u}_{M}^{BL})^{(n)}(x) \vert &\leq& C \hat{K}^n \varepsilon
_{2}^{-n}e^{-\beta (1-x)/\varepsilon _{2}} , \\
\Vert r_M \Vert_{\infty,\partial I}+\varepsilon_2 \Vert r'_M \Vert_{\infty,\partial I} &+& \Vert r_M \Vert_{1, I} +
\varepsilon_1 \Vert r''_M \Vert_{0, I} \leq C e^{-\delta / \varepsilon _{2}},
\end{eqnarray*}%
where $M$ is chosen so that $\varepsilon_2 e^2 4M \max \{ \tilde{\gamma}, \bar{\gamma}, \check{\gamma}, \hat{\gamma}, K_2 \} < 1$ and
$(\varepsilon_1 / \varepsilon_2^2) e^2 4M \max \{ \tilde{\gamma}, \bar{\gamma}, \check{\gamma}, \hat{\gamma}, K_2 \} < 1$. The constant $\beta$ is given by $\beta = \min \left\{ \sqrt{\tilde{c}_{0} / \tilde{b}_{0}},\sqrt{\bar{b}_{0}} ,
\sqrt{\check{c}_{0} / \check{b}_{0}},\sqrt{\hat{b}_{0}} \right\}$.
\end{theorem}
\begin{proof}
From (\ref{uM}) and Lemma \ref{lem:u_n_ij} we have
\begin{eqnarray*}
\Vert u^{(n)}_M \Vert_{\infty, I} &\leq& \sum_{i=0}^M \sum_{j=0}^M \varepsilon_2^i (\varepsilon_1 / \varepsilon_2^2)^j
\Vert u^{(n)}_{i,j} \Vert_{\infty, I} \leq C \sum_{i=0}^M \sum_{j=0}^M \varepsilon_2^i (\varepsilon_1 / \varepsilon_2^2)^j  n ! K_1^n i^i K_2^i \\
&\leq& C K_1^n n! \left( \sum_{i=0}^M \varepsilon_2^i i^i K_2^i\right) \left( \sum_{j=0}^M (\varepsilon_1 / \varepsilon_2^2)^j\right) \\
&\leq& C K_1^n n! \left( \sum_{i=0}^{\infty} (\varepsilon_2 M K_2)^i\right) \left( \sum_{j=0}^{\infty} (\varepsilon_1 / \varepsilon_2^2)^j \right) \\
&\leq& C K_1^n n! ,
\end{eqnarray*}
since both sums are convergent geometric series due to the assumptions $\varepsilon_2 M K_2 < 1$ and 
$\varepsilon_1 < \varepsilon_2^2$. 

We next show the result for $\tilde{u}_M^{BL}$. We have from (\ref{utM}) and Lemma \ref{lem:Dut_ij}
\begin{eqnarray*}
\vert (\tilde{u}^{BL}_M)^{(n)} (\tilde{x}) \vert  &\leq& \sum_{i=0}^M \sum_{j=0}^M \varepsilon_2^i (\varepsilon_1 / \varepsilon_2^2)^j \left\vert \left(\tilde{u}^{BL}_{i,j}\right)^{(n)}(\tilde{x}) \right\vert \\
&\leq& C \tilde{K}^n  \sum_{i=0}^M \sum_{j=0}^M \varepsilon_2^i (\varepsilon_1 / \varepsilon_2^2)^j   \tilde{\gamma}^{i+j} \frac{(2i+j)^{2i+j}}{i!} e^{-\beta \tilde{x} } .
\end{eqnarray*}
Using the inequality
\[
(2i+j)^{2i+j} \leq (2i)^{2i} e^{2i} j^j e^j ,
\]
we get 
\begin{eqnarray*}
\vert (\tilde{u}^{BL}_M)^{(n)} (\tilde{x}) \vert  &\leq& C (\tilde{K} e)^n e^{-\beta \tilde{x}} \left( \sum_{i=0}^M \frac{\varepsilon_2^i (2i)^{2i} \tilde{\gamma}^i e^{2i}}{i!} \right) \left( \sum_{j=0}^M (\varepsilon_1 / \varepsilon_2^2)^j j^j \tilde{\gamma}^j e^j \right) \\
&\leq& C (\tilde{K} e)^n e^{-\beta \tilde{x}} \left( \sum_{i=0}^{\infty} (\varepsilon_2 4M \tilde{\gamma} e^2)^i \right) \left( \sum_{j=0}^{\infty} ((\varepsilon_1 / \varepsilon_2^2) M \tilde{\gamma} e)^j \right) \\
&\leq& C (\tilde{K} e)^n e^{-\beta \tilde{x}},
\end{eqnarray*}
since both sums are convergent geometric series due to the assumptions $\varepsilon_2 4M \tilde{\gamma} e^2 < 1$ and $(\varepsilon_1 / \varepsilon_2^2) M \tilde{\gamma} e < 1$. The result for $\tilde{u}_{M}^{BL}$ follows by adjusting the constant $\tilde{K}$.

Similarly, we show the result for $\bar{u}_{M}^{BL}$. By (\ref{ubM}) and Lemma \ref{lem:Dut_ij}, we have
\begin{eqnarray*}
\vert (\bar{u}^{BL}_M)^{(n)} (\bar{x}) \vert  &\leq& \sum_{i=0}^M \sum_{j=0}^M \varepsilon_2^i (\varepsilon_1 / \varepsilon_2^2)^j \vert \bar{u}^{(n)}_{i,j}(\bar{x}) \vert \\
&\leq& C \bar{K}^n \sum_{i=0}^M \sum_{j=0}^M \varepsilon_2^i (\varepsilon_1 / \varepsilon_2^2)^j \bar{\gamma}^{i+j} \frac{(i+2j)^{i+2j}}{j!} e^{-\beta \bar{x} } .
\end{eqnarray*}
Using the inequality
\[
(i+2j)^{i+2j}  \leq i^i e^i (2j)^{2j} e^{2j},
\]
we obtain 
\begin{eqnarray*}
\vert (\bar{u}^{BL}_M)^{(n)} (\bar{x}) \vert  &\leq& C \bar{K}^n e^{-\beta \bar{x}} \left( \sum_{i=0}^M \varepsilon_2^i i^i \bar{\gamma}^i e^i \right) \left( \sum_{j=0}^M (\varepsilon_1 / \varepsilon_2^2)^j \frac{(2j)^{2j}}{j!} \bar{\gamma}^j e^{2j} \right) \\
&\leq& C \bar{K}^n e^{-\beta \bar{x}} \left( \sum_{i=0}^{\infty} (\varepsilon_2 M \bar{\gamma} e)^i \right) \left( \sum_{j=0}^{\infty} ((\varepsilon_1 / \varepsilon_2^2) 4M \bar{\gamma} e^2)^j \right) \\
&\leq& C \bar{K}^n e^{-\beta \bar{x}},
\end{eqnarray*}
where we used the assumptions on $M$ to see that the geometric series are bounded.

Analogous results hold for the functions $\check{u}_M^{BL}(\check{x})$ and $\hat{u}_M^{BL}(\hat{x})$, hence it remains to show the bound for the remainder, $r_M$. We note that
\begin{eqnarray*}
r_M(0) &=& u(0) - \left[ \sum_{i=0}^{M} \sum_{j=0}^{M} \varepsilon _{2}^{i} \left( \varepsilon _{1} / \varepsilon _{2}^{2} \right)^{j} \left( u_{i,j}(0) + \tilde{u}^{BL}_{i,j}(0) + \bar{u}^{BL}_{i,j}(0) + \check{u}^{BL}_{i,j}\left( \frac{1}{\varepsilon _{2}} \right) + \hat{u}^{BL}_{i,j}\left( \frac{\varepsilon _{2}}{\varepsilon _{1}} \right) \right) \right] \\
&=& - \sum_{i=0}^{M} \sum_{j=0}^{M} \varepsilon _{2}^{i} \left( \varepsilon _{1} / \varepsilon _{2}^{2} \right)^{j} \left( \check{u}^{BL}_{i,j}\left( \frac{1}{\varepsilon _{2}} \right) + \hat{u}^{BL}_{i,j}\left( \frac{\varepsilon _{2}}{\varepsilon _{1}} \right) \right)
\end{eqnarray*}
and we further get
\begin{eqnarray*}
|r_M(0)| &\leq& \sum_{i=0}^{M} \sum_{j=0}^{M} \varepsilon _{2}^{i} \left( \varepsilon _{1} / \varepsilon _{2}^{2} \right)^{j} \left[ \left\vert \check{u}^{BL}_{i,j}\left( \frac{1}{\varepsilon _{2}} \right) \right\vert + \left\vert \hat{u}^{BL}_{i,j}\left( \frac{\varepsilon _{2}}{\varepsilon _{1}} \right) \right\vert \right] \\
&\leq& C \sum_{i=0}^{M} \sum_{j=0}^{M} \varepsilon _{2}^{i} \left( \varepsilon _{1} / \varepsilon _{2}^{2} \right)^{j} \left[ \check{\gamma}^{i+j} \frac{(2i+j)^{2i+j}}{i!} e^{-\beta / \varepsilon _{2}} + \hat{\gamma}^{i+j} \frac{(i+2j)^{i+2j}}{j!} e^{-\beta \varepsilon _{2} / \varepsilon _{1}} \right] \\
&\leq& C e^{-\beta / \varepsilon _{2}} \left( \sum_{i=0}^{\infty} (\varepsilon_2 4M \check{\gamma} e^2)^i \right) \left( \sum_{j=0}^{\infty} ((\varepsilon_1 / \varepsilon_2^2) M \check{\gamma} e)^j \right) + \\
&+& C e^{-\beta \varepsilon _{2} / \varepsilon _{1}} \left( \sum_{i=0}^{\infty} (\varepsilon_2 M \hat{\gamma} e)^i \right) \left( \sum_{j=0}^{\infty} ((\varepsilon_1 / \varepsilon_2^2) 4M \hat{\gamma} e^2)^j \right) \\
&\leq& C \left\{ e^{-\beta / \varepsilon _{2}} + e^{-\beta \varepsilon _{2} / \varepsilon _{1}} \right\}.
\end{eqnarray*}
Similarly,
\begin{eqnarray*}
|r_M(1)| &\leq& \sum_{i=0}^{M} \sum_{j=0}^{M} \varepsilon _{2}^{i} \left( \varepsilon _{1} / \varepsilon _{2}^{2} \right)^{j} \left[ \left\vert \tilde{u}^{BL}_{i,j}\left( \frac{1}{\varepsilon _{2}} \right) \right\vert + \left\vert \bar{u}^{BL}_{i,j}\left( \frac{\varepsilon _{2}}{\varepsilon _{1}} \right) \right\vert \right] \\
&\leq& C \sum_{i=0}^{M} \sum_{j=0}^{M} \varepsilon _{2}^{i} \left( \varepsilon _{1} / \varepsilon _{2}^{2} \right)^{j} \left[ \tilde{\gamma}^{i+j} \frac{(2i+j)^{2i+j}}{i!} e^{-\beta / \varepsilon _{2}} + \bar{\gamma}^{i+j} \frac{(i+2j)^{i+2j}}{j!} e^{-\beta \varepsilon _{2} / \varepsilon _{1}} \right] \\
&\leq& C e^{-\beta / \varepsilon _{2}} \left( \sum_{i=0}^{\infty} (\varepsilon_2 4M \tilde{\gamma} e^2)^i \right) \left( \sum_{j=0}^{\infty} ((\varepsilon_1 / \varepsilon_2^2) M \tilde{\gamma} e)^j \right) + \\
&+& C e^{-\beta \varepsilon _{2} / \varepsilon _{1}} \left( \sum_{i=0}^{\infty} (\varepsilon_2 M \bar{\gamma} e)^i \right) \left( \sum_{j=0}^{\infty} ((\varepsilon_1 / \varepsilon_2^2) 4M \bar{\gamma} e^2)^j \right) \\
&\leq& C \left\{ e^{-\beta / \varepsilon _{2}} + e^{-\beta \varepsilon _{2} / \varepsilon _{1}} \right\}.
\end{eqnarray*}
Combining the two results, we have
\begin{equation*}
\left\Vert r_M \right\Vert _{\infty ,\partial I} \leq C \max \{ e^{-\beta / \varepsilon _{2}},e^{-\beta \varepsilon _{2} / \varepsilon _{1}} \}.
\end{equation*}
In a completely analogous way, we may show
\begin{equation*}
\left\Vert r_M^{\prime } \right\Vert _{\infty ,\partial I} \leq C \varepsilon _2^{-1} \max \{ e^{-\beta / \varepsilon _{2}},e^{-\beta \varepsilon _{2} / \varepsilon _{1}} \}.
\end{equation*}
Now, the remainder $r_M$ satisfies
\begin{eqnarray*}
L_{\varepsilon_1, \varepsilon_2} r_M &=& L_{\varepsilon_1, \varepsilon_2} u - L_{\varepsilon_1, \varepsilon_2} u_M 
- L_{\varepsilon_1, \varepsilon_2} \left( \tilde{u}_M^{BL} + \bar{u}_M^{BL} + \check{u}_M^{BL} + \hat{u}_{M}^{BL} \right)  \\
&=& f - L_{\varepsilon_1, \varepsilon_2} u_M -L_{\varepsilon_1, \varepsilon_2} \left( \tilde{u}_M^{BL} + \bar{u}_M^{BL} + \check{u}_M^{BL} + \hat{u}_{M}^{BL} \right).
\end{eqnarray*}
We first consider
\[
f - L_{\varepsilon_1, \varepsilon_2} u_M = f - \sum_{i=0}^M \sum_{j=0}^M \varepsilon_2^i (\varepsilon_1 / \varepsilon_2^2)^j
L_{\varepsilon_1, \varepsilon_2} u_{i,j}
\]
with $u_{i,j}$ satisfying (\ref{uij}). After some calculations, we find
\[
f - L_{\varepsilon_1, \varepsilon_2} u_M = - \varepsilon_2^{M+1} \sum_{j=2}^{M+2} (\varepsilon_1 / \varepsilon_2^2)^j u_{M-3,j-2}^{(4)} - \varepsilon_2^{M+2} \sum_{j=2}^{M+2} (\varepsilon_1 / \varepsilon_2^2)^j u_{M-2,j-2}^{(4)} -
\]%
\[
- \varepsilon_2^{M+3} \sum_{j=2}^{M+2} (\varepsilon_1 / \varepsilon_2^2)^j u_{M-1,j-2}^{(4)} - \varepsilon_2^{M+4} \sum_{j=2}^{M+2} (\varepsilon_1 / \varepsilon_2^2)^j u_{M,j-2}^{(4)} - (\varepsilon_1 / \varepsilon_2^2)^{M+1} \sum_{i=4}^M \varepsilon_2^i u_{i-4,M-1}^{(4)} -
\]%
\[
- (\varepsilon_1 / \varepsilon_2^2)^{M+2} \sum_{i=4}^M \varepsilon_2^i u_{i-4,M}^{(4)} + \varepsilon_2^{M+1} \sum_{j=0}^M (\varepsilon_1 / \varepsilon_2^2)^j \left( b u_{M-1,j}^{\prime } \right) ^{\prime } + \varepsilon_2^{M+2} \sum_{j=0}^M (\varepsilon_1 / \varepsilon_2^2)^j \left( b u_{M,j}^{\prime } \right) ^{\prime }.
\]%
Hence,
\begin{eqnarray*}
\Vert f - L_{\varepsilon_1, \varepsilon_2} u_M \Vert _{\infty ,I} &\lesssim& \varepsilon_2^{M+1} \sum_{j=2}^{M+2} (\varepsilon_1 / \varepsilon_2^2)^j \Vert u_{M-3,j-2}^{(4)} \Vert _{\infty ,I} + (\varepsilon_1 / \varepsilon_2^2)^{M+1} \sum_{i=4}^M \varepsilon_2^i \Vert u_{i-4,M-1}^{(4)} \Vert _{\infty ,I} + \\
&+& \varepsilon_2^{M+1} \sum_{j=0}^M (\varepsilon_1 / \varepsilon_2^2)^j \Vert b \Vert _{\infty ,I} \Vert u_{M-1,j}^{\prime } \Vert _{\infty ,I},
\end{eqnarray*}
where we used Cauchy's Integral Theorem for Derivatives. Using Lemma \ref{lem:u_n_ij}, we further obtain
\begin{eqnarray*}
\Vert f - L_{\varepsilon_1, \varepsilon_2} u_M \Vert _{\infty ,I} &\lesssim& \varepsilon_2^{M+1} (M-3)^{M-3} K_2^{M-3} \sum_{j=2}^{M+2} (\varepsilon_1 / \varepsilon_2^2)^j +  \\
&+& (\varepsilon_1 / \varepsilon_2^2)^{M+1} (M-4)^{M-4} K_2^{M-4} \sum_{i=4}^M \varepsilon_2^i + \\
&+& \varepsilon_2^{M+1} (M-1)^{M-1} K_2^{M-1} \sum_{j=0}^M (\varepsilon_1 / \varepsilon_2^2)^j \\
&\lesssim& \varepsilon_2 (\varepsilon_2 M K_2)^M + (\varepsilon_1 / \varepsilon_2^2) \left( (\varepsilon_1 / \varepsilon_2^2) M K_2 \right) ^M,
\end{eqnarray*}
since the finite sums can be bounded by converging geometric series.

Next, we consider $L_{\varepsilon_1, \varepsilon_2} \tilde{u}_M^{BL}$ and using (\ref{ut_ij}) and Lemma \ref{lem:Dut_ij}, we have
\begin{eqnarray*}
\Vert L_{\varepsilon_1, \varepsilon_2} \tilde{u}_M^{BL} \Vert _{\infty ,I} &=& \sum_{i=0}^M \sum_{j=0}^M \varepsilon_2^i (\varepsilon_1 / \varepsilon_2^2)^j \Vert L_{\varepsilon_1, \varepsilon_2} \tilde{u}^{BL}_{i,j} \Vert _{\infty ,I} \\
&\lesssim& (\varepsilon_1 / \varepsilon_2^2)^{M+1} \sum_{i=0}^M \varepsilon_2^i \Vert \left(\tilde{u}^{BL}_{i,M-1}\right)^{(4)} \Vert _{\infty ,I} \\
&\lesssim& (\varepsilon_1 / \varepsilon_2^2)^{M+1} \sum_{i=0}^M \varepsilon_2^i \tilde{\gamma }^{i+M-1} \frac{(2i+M-1)^{2i+M-1}}{i!} \\
&\lesssim& (\varepsilon_1 / \varepsilon_2^2)^{M+1} \sum_{i=0}^M \varepsilon_2^i \tilde{\gamma }^{i+M-1} (4i)^{i} e^{2i} (M-1)^{M-1} e^{M-1} \\
&\lesssim& \left( (\varepsilon_1 / \varepsilon_2^2) \tilde{\gamma } e M \right) ^{M+1} \sum_{i=0}^M (\varepsilon_2 4M \tilde{\gamma } e^2)^i \\
&\lesssim& \left( (\varepsilon_1 / \varepsilon_2^2) \tilde{\gamma } e M \right) ^{M+1}.
%&+& \varepsilon_2^{M+1} \sum_{\ell =1}^M \left\{ \Vert \tilde{b}_{\ell } \Vert _{\infty ,I} \Vert \tilde{u}_{M+1-\ell ,0}^{\prime } \Vert _{\infty ,I} + \Vert \tilde{c}_{\ell } \Vert _{\infty ,I} \Vert \tilde{u}_{M+1-\ell ,0} \Vert _{\infty ,I} \right\} \\
\end{eqnarray*}
Similarly, for $L_{\varepsilon_1, \varepsilon_2} \bar{u}_M^{BL}$ we have
\begin{eqnarray*}
\Vert L_{\varepsilon_1, \varepsilon_2} \bar{u}_M^{BL} \Vert _{\infty ,I} &=& \sum_{i=0}^M \sum_{j=0}^M \varepsilon_2^i (\varepsilon_1 / \varepsilon_2^2)^j \Vert L_{\varepsilon_1, \varepsilon_2} \bar{u}^{BL}_{i,j} \Vert _{\infty ,I} \\
&\lesssim& \varepsilon_2^{M+1} \sum_{j=0}^M \left( \frac{\varepsilon_1}{\varepsilon_2^2}\right)^j \Vert \left(\bar{u}_{M-1,j}^{BL}\right)^{(4)}\Vert_{\infty,I} \\
&\lesssim& \varepsilon_2^{M+1} \sum_{j=0}^M \left( \frac{\varepsilon_1}{\varepsilon_2^2}\right)^j 
 \bar{\gamma }^{j+M-1} \frac{(2j+M-1)^{2j+M-1}}{j!} \\
&\lesssim& \varepsilon_2^{M+1} \sum_{j=0}^M \left( \frac{\varepsilon_1}{\varepsilon_2^2}\right)^j 
 \bar{\gamma }^{j+M-1} (4j)^{j} e^{2j} (M-1)^{M-1} e^{M-1} \\
&\lesssim& \left( \varepsilon_2 \bar{\gamma } e M \right) ^{M+1} \sum_{j=0}^M \left(\frac{\varepsilon_1}{\varepsilon^2_2} 4M \bar{\gamma } e^2\right)^j\\
&\lesssim& \left( \varepsilon_2 \bar{\gamma } e M \right) ^{M+1}.
\end{eqnarray*}
Similar results hold for the functions $\check{u}^{BL}_{M}, \hat{u}^{BL}_{M}$. Therefore, combining all of the above we have
\begin{eqnarray*}
\Vert L_{\varepsilon_1, \varepsilon_2} r_M \Vert _{\infty ,I} &=& \Vert f - L_{\varepsilon_1, \varepsilon_2} u_M -L_{\varepsilon_1, \varepsilon_2} \left( \tilde{u}_M^{BL} + \bar{u}_M^{BL} + \check{u}_M^{BL} + \hat{u}_{M}^{BL} \right) \Vert _{\infty ,I} \\
&\leq& \Vert f - L_{\varepsilon_1, \varepsilon_2} u_M \Vert _{\infty ,I} + \Vert L_{\varepsilon_1, \varepsilon_2} \tilde{u}_M^{BL} \Vert _{\infty ,I} + \Vert L_{\varepsilon_1, \varepsilon_2} \bar{u}_M^{BL} \Vert _{\infty ,I} + \\
&+& \Vert L_{\varepsilon_1, \varepsilon_2} \check{u}_M^{BL} \Vert _{\infty ,I} + \Vert L_{\varepsilon_1, \varepsilon_2} \hat{u}_M^{BL} \Vert _{\infty ,I} \\
&\lesssim&  \varepsilon_2 (\varepsilon_2 M K_2)^M + (\varepsilon_1 / \varepsilon_2^2) \left( (\varepsilon_1 / \varepsilon_2^2) M K_2 \right) ^M +
 \left( \varepsilon_2 \bar{\gamma } e M \right) ^{M+1} \\
&+& \left( (\varepsilon_1 / \varepsilon_2^2) \hat{\gamma } e M \right) ^{M+1} + \left( \varepsilon_2 \hat{\gamma } e M \right) ^{M+1},
\end{eqnarray*}
from which the desired result follows.
\end{proof}

We close with the following alternative regularity result which will be needed for the numerical analysis of a high order Finite Element Method.

\begin{theorem}\label{thm:main2}
Let $u$ be the solution to (\ref{de})--(\ref{bc}) and assume (\ref{analytic}) and (\ref{data}) hold. Then, there
exist positive constants $C,$ $K_1,$ $\tilde{K},$ $\hat{K},$  $\tilde{\gamma},$ $\hat{\gamma},$ $\delta ,$ independent of $\varepsilon_1, \varepsilon_2,$ such that $\forall \; n \in \mathbb{N}$ and $\forall \; x \in \overline{I},$ there holds
\begin{equation*}
u = u_M + \tilde{u}_M^{BL} + \hat{u}_M^{BL} + r_M,
\end{equation*}
and
\begin{eqnarray*}
\Vert u_M^{(n)} \Vert_{\infty, I} &\leq& C K_1^n \eps_2^{-1}\max \{ n, \eps^{-1}_2 \}^n  ,\\
 \vert (\tilde{u}^{BL}_{M})^{(n)}(x) \vert &\leq& C \tilde{K}^n \left( \frac{\varepsilon _{1}}{\varepsilon _{2}}\right) ^{-n}e^{-\beta \varepsilon
_{2}x/\varepsilon _{1}},\\
\vert (\hat{u}_{M}^{BL})^{(n)}(x) \vert &\leq& C \hat{K}^n \varepsilon
_{2}^{-n}e^{-\beta (1-x)/\varepsilon _{2}} , \\
\eps_2^{-1}\left\Vert r_M\right\Vert _{0,\Omega }+\varepsilon_2 \Vert r'_M \Vert_{0,\Omega} + \varepsilon_2^3 \Vert r''_M \Vert_{0,\Omega} &\leq& {C}  e^{-\delta \eps_2 /\eps_1},
\end{eqnarray*}%
where $M$ is chosen so that 
$(\varepsilon_1 / \varepsilon_2^2) e^2 4M \max \{ \tilde{\gamma}, \hat{\gamma}\} < 1$. The constant $\beta$ is given by $\beta = \min \left\{ \sqrt{\tilde{c}_{0} / \tilde{b}_{0}}, \sqrt{\hat{c}_{0} / \hat{b}_{0}},\sqrt{\hat{b}_{0}} \right\}$.
\end{theorem}
\begin{proof}
We rewrite the BVP as
\begin{equation}\label{eq:bvp_sa}
\left.
\begin{gathered}
\mathcal{L}u:=\left(\frac{\varepsilon _{1}}{\varepsilon _{2}}\right)^2 u^{(4)}(x)-\left( b(x)u^{\prime
}(x)\right) ^{\prime }+\frac{c(x)}{\eps_2^2} u(x) =\frac{f(x)}{\eps_2^2}\ \text{in }\Omega =\left( 0,1\right)   \\
u(0)=u(1)=u^{\prime }(0)=u^{\prime }(1) =0
\end{gathered}
\right\},
\end{equation}
which can be viewed as a one-parameter problem, the parameter being $\eps_1 / \eps_2$, with data depending on $\eps_2$. In \cite{PC, PCarXiv} one parameter fourth order problems with analytic data (independent of $\eps_2$) were studied and analytic regularity results were obtained (see \cite[Thm. 7]{PCarXiv}). Therefore, we will revisit the proofs of the relevant results in \cite{PCarXiv} to see how these are altered by the presence of $\eps^{-2}_2$ in the data.

We define the stretched variables  $\tilde{x} = x \eps_2 / \eps_1 ,  \hat{x} = (1-x) \eps_2 / \eps_1$, and make the ansatz
$$
u = \sum_{i=0}^{\infty} \left( \frac{\eps_1}{\eps_2} \right)^i \{ U_i(x) + \tilde{U}_i(\tilde{x}) + \hat{U}_i (\hat{x})\},
$$
with the functions $U_i , \tilde{U}_i , \hat{U}_i $ to be determined. We substitute the above expression into the differential equation in (\ref{eq:bvp_sa}), separate the slow (i.e.  $x$) from the fast (i.e. $\tilde{x}, \hat{x}$ variables), and equate like powers of $\eps_1 / \eps_2$ on both sides. This gives information on the functions $U_i , \tilde{U}_i , \hat{U}_i $, as described in Table 1, in which we use the notation
$\tilde{b}_k(\tilde{x})=\frac{b^{(k)}(0)}{k!} \tilde{x}^k$, $\tilde{c}_k(\tilde{x})=\frac{c^{(k)}(0)}{k!} \tilde{x}^k$,  $\hat{b}_k(\tilde{x})=\frac{b^{(k)}(1)}{k!} \hat{x}^k$, $\hat{c}_k(\hat{x})=\frac{c^{(k)}(1)}{k!} \hat{x}^k$, and the functions $\tilde{F}_j$ are given by  (see \cite{PCarXiv} for more details)
\begin{equation}\label{Fj}
\tilde{F}_j(\tilde{x}):=\sum_{k=1}^{A_j} \tilde{b}_k(\tilde{x}) \tilde{U}_{j-k}^{\prime \prime}(\tilde{x})+\sum_{k=1}^{A_j} \tilde{b}_k^{\prime}(\tilde{x}) \tilde{U}_{j-k}^{\prime}(\tilde{x})-\sum_{k=0}^{B_j} \eps_2^{-2} \tilde{c}_k(\tilde{x}) \tilde{U}_{j-2-k}(\tilde{x}), 
\end{equation}
where
\begin{equation*}
\begin{aligned}
&A_j=\left\{\begin{array}{ll}
\frac{j}{2}, & \text { if } j \text { is even }, \\
\frac{j-1}{2}, & \text { if } j \text { is odd },
\end{array} \quad B_j= \begin{cases}\frac{j-2}{2}, & \text { if } j \text { is even }, \\
\frac{j-3}{2}, & \text { if } j \text { is odd } .\end{cases}\right. \notag
\end{aligned}
\end{equation*}
The functions $\hat{F}_j$ are defined in a similar fashion (see \cite{PCarXiv}).

\begin{table}[h]\label{table1}
\begin{center}
\begin{equation*}
\begin{array}{||l||l||l||}
\hline \text { Smooth part } & \text { Left boundary layer} & \text {Right boundary layer }\\
\hline \hline-\eps_2^2 \left(b U_0^{\prime}\right)^{\prime}+c U_0=f, & \tilde{U}_0^{(4)}-\left(\tilde{b}_0 \tilde{U}_0^{\prime}\right)^{\prime}=0, & \hat{U}_0^{(4)}-\left(\hat{b}_0 \hat{U}_0^{\prime}\right)^{\prime}=0, \\
U_0(0)=-\tilde{U}_0(0)=0, & \lim _{\tilde{x} \rightarrow \infty} \tilde{U}_0(\tilde{x})=0, & \lim _{\hat{x} \rightarrow \infty} \hat{U}_0(\hat{x})=0, \\
U_0(1)=-\hat{U}_0(0)=0, & \tilde{U}_0^{\prime}(0)=0, & \hat{U}_0^{\prime}(0)=0, \\
\hline \hline-\eps_2^2 \left(b U_1^{\prime}\right)^{\prime}+c U_1=0, & \tilde{U}_1^{(4)}-\left(\tilde{b}_0 \tilde{U}_1^{\prime}\right)^{\prime}=0, & \hat{U}_1^{(4)}-\left(\hat{b}_0 \hat{U}_1^{\prime}\right)^{\prime}=0, \\
U_1(0)=-\tilde{U}_1(0), & \lim _{\tilde{x} \rightarrow \infty} \tilde{U}_1(\tilde{x})=0, & \lim _{\hat{x} \rightarrow \infty} \hat{U}_1(\hat{x})=0, \\
U_1(1)=-\hat{U}_1(0), & \tilde{U}_1^{\prime}(0)=-U_0^{\prime}(0), & \hat{U}_1^{\prime}(0)=-U_0^{\prime}(1), \\
\hline \hline-\eps_2^2\left(b U_j^{\prime}\right)'+c U_j=-U_{j-2}^{(4)}, & \tilde{U}_j^{(4)}-\left(\tilde{b}_0 \tilde{U}_j^{\prime}\right)'=\tilde{F}_j, & \hat{U}_j^{(4)}-\left(\hat{b}_0 \hat{U}_j^{\prime}\right)^{\prime}=\hat{F}_j, \\
U_j(0)=-\tilde{U}_j(0), & \lim _{\tilde{x} \rightarrow \infty} \tilde{U}_j(\tilde{x})=0, & \lim _{\hat{x} \rightarrow \infty} \hat{U}_j(\hat{x})=0, \\
U_j(1)=-\hat{U}_j(0), & \tilde{U}_j^{\prime}(0)=-U_{j-1}^{\prime}(0), & \hat{U}_j^{\prime}(0)=-U_{j-1}^{\prime}(1) . \\
\hline
\end{array}
\end{equation*}
\end{center}
\caption{Different boundary value problems satisfied by $U_j, \tilde{U}_j, \hat{U}_j , j \in \mathbb{N}_0$.}
\end{table}
Due to the BCs, the sequences $\{U_j, \tilde{U}_j,  \hat{U}_j\}_{j\ge 1}$ are intertwined, so we will treat them simultaneously. For $j=0$, we have (see, e.g., \cite[Prop. 2.3.1]{PC} or \cite{melenk})
$$
\Vert U_0^{(n)}\Vert_{\infty, \Omega} \leq C_{0} K_{0}^{n} \max\{ n, \eps_2^{-1}\}^{n}.
$$
We also calculate $\tilde{U}_0 = 0, \hat{U}_0 = 0$. For $j=1$, we find
$$
\tilde{U}_1(\tilde{x}) = \frac{U'_0(0)}{\sqrt{b(0)}} e^{-\sqrt{b(0)} \tilde{x}} \; , \; 
\hat{U}_1(\hat{x}) = \frac{U'_0(1)}{\sqrt{b(1)}} e^{-\sqrt{b(1)} \hat{x}}.
$$
Thus $U_1$ satisfies
$$
 -\eps_2^2 \left(b U_1^{\prime}\right)^{\prime}+c U_1=0,
$$
with boundary conditions that are $O(\eps_2^{-1})$. Hence, (see, e.g., \cite{melenk})
$$
\Vert U^{(n)}_1 \Vert_{\infty, \Omega} \leq C \eps_2^{-1} \max \{n, \eps_2^{-1} \}^n.
$$
%We also have that $\tilde{U}_2(\tilde{x})$ and $\hat{U}_2(\hat{x})$ satisfy
%$$
%\tilde{U}_2^{(4)}-\left(\tilde{b}_0 \tilde{U}_2^{\prime}\right)'=\tilde{F}_2\; , \; \hat{U}_2^{(4)}-\left(\hat{b}_0 \hat{U}_2^{\prime}\right)^{\prime}=\hat{F}_2,
%$$
%
%with boundary conditions that are $O(\eps_2^{-1})$, where
%
%\begin{equation*}
%\tilde{F}_2(\tilde{x})= \tilde{b}_1(\tilde{x}) \tilde{U}_{1}^{\prime \prime}(\tilde{x})+ \tilde{b}_1^{\prime}(\tilde{x}) \tilde{U}_{1}^{\prime}(\tilde{x}),
%\end{equation*}
%and similarly for $\hat{F}_2$. There holds
%
%$$
%\vert \tilde{F}_2 \vert \leq C \eps_2^{-1} \; , \;\vert \hat{F}_2 \vert \leq C \eps_2^{-1}
%$$
%and as a result (see the proof of \cite[Thm. 2.4.16]{PC} for the steps)
%$$
%\vert \tilde{U}_2 (\tilde{x}) \vert \leq C \eps_2^{-2}  e^{-\beta \tilde{x}} .
%$$
%
%Now, $U_2$ satisfies
%
%$$
%-\eps_2^2 (bU_2')' +c U_2 = - U^{(4)}_0,
%$$
%with boundary conditions that are $O(\eps_2^{-2})$. Hence,
%
%$$
%\Vert U_2 \Vert_{\infty,\Omega} \leq C \eps_2^{-4}.
%$$
%
\textit{Claim} : There exist positive constants $C, K_1, K_2, \tilde{K}_1, \hat{K}_1, \alpha$, independent of
$\eps_1, \eps_2$ such that 
\begin{equation}\label{claim_a}
\Vert U_j^{(n)}\Vert_{\infty, \Omega} \leq C K_{1}^{n} K_2^j j! \eps_2^{-j} \max\{ n, \eps_2^{-1} \}^{n} \; \forall \; j \in \mathbb{N}_0, n \in \mathbb{N}_0,
\end{equation}
\begin{equation}\label{claim_b}
\vert \tilde{U}_j (z) \vert \leq C  \tilde{K}^j_1  \eps_2^{-j} (\alpha j +|z|)^{2j-1} e^{-\text{Re}(z)} \; \forall \; j \in \mathbb{N}, z \in \mathbb{C},
%
%\vert \hat{U}_j (\tilde{x}) \vert \leq C  \hat{K}^j_1,
\end{equation}
and similarly for $\hat{U}_j$.

To show the claim we will use induction on $j$ with the base cases shown above. So we assume the results
hold for $j$, and establish them for $j+1$. First, we note that the functions $U_{j+1}, \tilde{U}_{j+1}, \hat{U}_{j+1}$ satisfy
\begin{equation}\label{eq:Ujp1}
\left.
\begin{gathered}
-\eps_2^2\left(b U_{j+1}^{\prime}\right)'+c U_{j+1}=-U_{j-2}^{(4)} \\
U_{j+1}(0)=-\tilde{U}_{j+1}(0) \; , \; U_{j+1}(1)=-\hat{U}_{j+1}(0)
\end{gathered}
\right\},
\end{equation}
\begin{equation}\label{eq:Ujtp1}
\left.
\begin{gathered}
\tilde{U}_{j+1}^{(4)}-\left(\tilde{b}_0 \tilde{U}_{j+1}^{\prime}\right)' = \tilde{F}_{j+1} \\
 \lim _{\tilde{x} \rightarrow \infty} \tilde{U}_{j+1}(\tilde{x})=0 \; , \; \tilde{U}_{j+1}^{\prime}(0)=-U_{j}^{\prime}(0)
\end{gathered}
\right\},
\end{equation}
\begin{equation}\label{eq:Ujhp1}
\left.
\begin{gathered}
\hat{U}_{j+1}^{(4)}-\left(\hat{b}_0 \hat{U}_{j+1}^{\prime}\right)^{\prime}=\hat{F}_{j+1} \\
\lim _{\hat{x} \rightarrow \infty} \hat{U}_{j+1}(\hat{x})=0 \; \, \;  \hat{U}_{j+1}^{\prime}(0)=-U_{j}^{\prime}(1)
\end{gathered}
\right\}.
\end{equation}
The right hand side $\tilde{F}_{j+1}$ in (\ref{eq:Ujtp1}) (cf.~eq.~(\ref{Fj})), satisfies 
$$
\vert \tilde{F}_{j+1}(z) \vert \leq C \gamma^j \eps_2^{-j} (\alpha j + |z| )^{2j-1},
$$
where the induction hypothesis was used (see \cite{PCarXiv} for details). Then, by \cite[Lemma 1]{PCarXiv}, we have (using also
that $|\tilde{U}'_{j+1}(0) | = O(\eps_2^{-j})$),
$$
\vert \tilde{U}_{j+1} (z) \vert \leq C  \tilde{K}^{j+1}_1  \eps_2^{-(j+1)} (\alpha (j+1) +|z|)^{2j} e^{-\text{Re}(z)},
$$
from which (\ref{claim_b}) follows.

To show (\ref{claim_a}), note that the function $U_{j+1}$, satisfies (\ref{eq:Ujp1}) with the right hand side and the boundary conditions being $O\left(\eps_2^{-(j+1)}\right)$ and $O\left(\eps_2^{-j}\right)$, respectively. Theorem 1 in \cite{melenk} gives
$$
\Vert U^{(n)}_{j+1} \Vert_{\infty,\Omega} \leq C K_1^n K_2^{j+1} \eps_2^{-(j+1)} \max\{n, \eps_2^{-1} \}^n,
$$
hence the claim.

Now, using Cauchy's integral theorem for derivatives, we may show that the function $\tilde{U}_j$ satisfies
\begin{eqnarray*}
\left|\tilde{U}_j^{(n)}({x})\right| &\leq& C \tilde{K}^n \left( \frac{\eps_1}{\eps_2} \right)^{-n} \left(\alpha^2 e \gamma\right)^j j^{j-1} e^{-\frac{\sqrt{b(0)}}{2} \eps_2 {x} / \eps_1} \\
&\leq& C \left( \frac{\eps_1}{\eps_2} \right)^{-n} \tilde{K}^n \tilde{K}_1^{j} j!  e^{-\beta \eps_2 {x} / \eps_1},
\end{eqnarray*}
and similarly for $ \hat{U}_j$. (For the details see \cite[Corollary 1]{PCarXiv}.)

Now, with $M \in \mathbb{N}$, we define the smooth part $S$ as
$$
S(x) = \sum_{j=0}^M \left( \frac{\eps_1}{\eps_2} \right)^j U_j(x),
$$
the left boundary layer $E_1$, as
$$
E_1(\tilde{x}) = \sum_{j=0}^M \left( \frac{\eps_1}{\eps_2} \right)^j \tilde{U}_j(\tilde{x}),
$$
and similarly for the right boundary layer $E_2(\hat{x})$. 

Therefore,
\begin{eqnarray*}
\Vert S^{(n)} \Vert_{\infty,\Omega} &\leq& \sum_{j=0}^M \left( \frac{\eps_1}{\eps_2} \right)^j \Vert U_j^{(n)} \Vert_{\infty,\Omega} \\
&\leq& C K_1^n  \max\{ n, \eps_2^{-1} \}^{n} \sum_{j=0}^M \left( \frac{\eps_1}{\eps_2} \right)^j j! K_{2}^{j} \eps_2^{-j}  \\
&\leq& C K_1^n  \max\{ n, \eps_2^{-1} \}^{n} \sum_{j=0}^{\infty} \left( \frac{\eps_1}{\eps^2_2} M K_{2}  \right)^j .
\end{eqnarray*}
If we select $M \in \mathbb{N}$ such that $\frac{\eps_1}{\eps^2_2} M K_{2} < 1$ then the geometric series converges and we have the result for $S$.

Similarly,
\begin{eqnarray*}
\vert E^{(n)}_1(\tilde{x}) \vert &\leq& \sum_{j=0}^M \left( \frac{\eps_1}{\eps_2} \right)^j \vert \tilde{U}^{(n)}_j(\tilde{x}) \vert \\
&\leq&C  \tilde{K}_{2}^{n} \left( \frac{\eps_1}{\eps_2} \right)^{-n} e^{-\tilde{x}} \sum_{j=0}^M \left( \frac{\eps_1}{\eps_2} \right)^j   \tilde{K}^j_1 j!\eps_2^{-j} ,
\end{eqnarray*}
with the rest of the steps being the same as above. The result holds true for $E_2$ as well, by symmetry.

It remains to deal with the remainder, $R = u - S - E_1 - E_2$, and we have
\begin{eqnarray*}
\mathcal{L} R &=& \mathcal{L}u - \mathcal{L}S - \mathcal{L}E_1 - \mathcal{L}E_2 = \frac{f}{\eps_2^2} - \sum_{i=0}^{M} \left( \frac{\eps_1}{\eps_2} \right)^i \mathcal{L}U_i \\
&=& \frac{f}{\eps_2^2} - \sum_{i=0}^{M} \left( \frac{\eps_1}{\eps_2} \right)^i \left\{\left(\frac{\varepsilon _{1}}{\varepsilon _{2}}\right)^2 U_i^{(4)}(x)-\left( b(x)U_i^{\prime
}(x)\right) ^{\prime }+\frac{c(x)}{\eps_2^2} U_i(x) \right\} \\
&=& \left( \frac{\eps_1}{\eps_2} \right)^{M+2} U_M^{(4)},
\end{eqnarray*}
due to cancellations. A standard energy argument gives
$$
\eps_2^3 \Vert R'' \Vert_{0,\Omega} + \eps_2 \Vert R' \Vert_{0,\Omega}+ \eps_2^{-1} \Vert R \Vert_{0,\Omega} \leq C \left( \frac{\eps_1}{\eps^2_2} K_2 M \right)^M,
$$
and since we chose $M$ such that the quantity in parentheses is smaller than $1$, we have the desired result. This completes the proof of Theorem \ref{thm:main2}.

\end{proof}

%%%%%%%%%%%%%%%%%%%%%%%%%%%%%%%%%%%%%%%%%%%%%%%%%
\section{Numerical illustration\label{nr}}
In this section we show the exact solution to the problem, for the constant coefficient case, in order to visualize the layers present. We choose $b(x)=c(x)=f(x)=1$, and we may obtain an exact solution. In the figures below we plot $u$ and $u'$, for various values of $\eps_1, \eps_2$. 

First, in Figure \ref{F1}, we show the solution (left) and its derivative (right) for $\eps_1 = \eps_2 = 10^{-2}$. We see that, as expected, the solution $u$ does not feature any layers, while $u'$ does.

\begin{figure}[h]
\begin{center}
\includegraphics[width=0.4\textwidth]{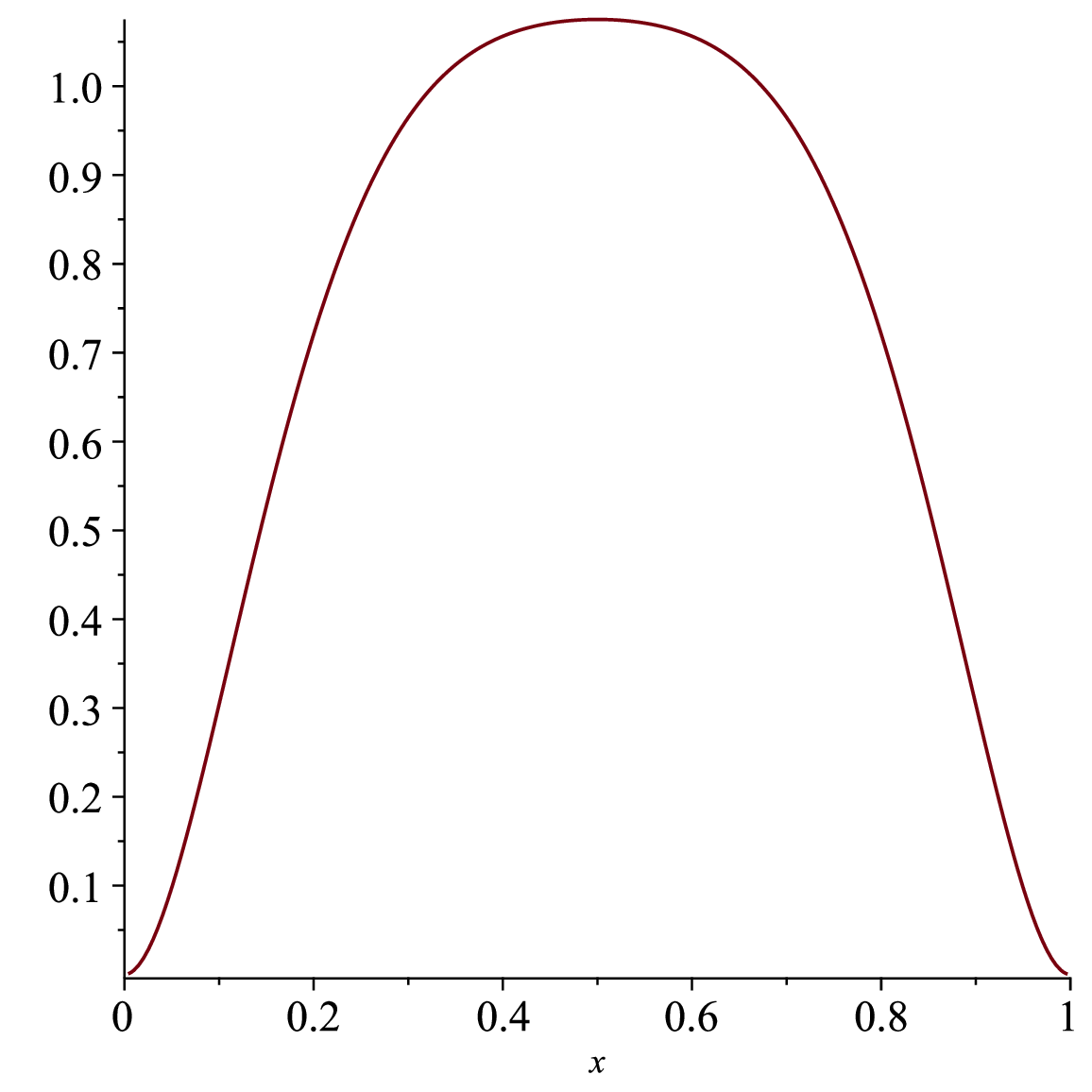} 
\mbox{ }
\includegraphics[width=0.4\textwidth]{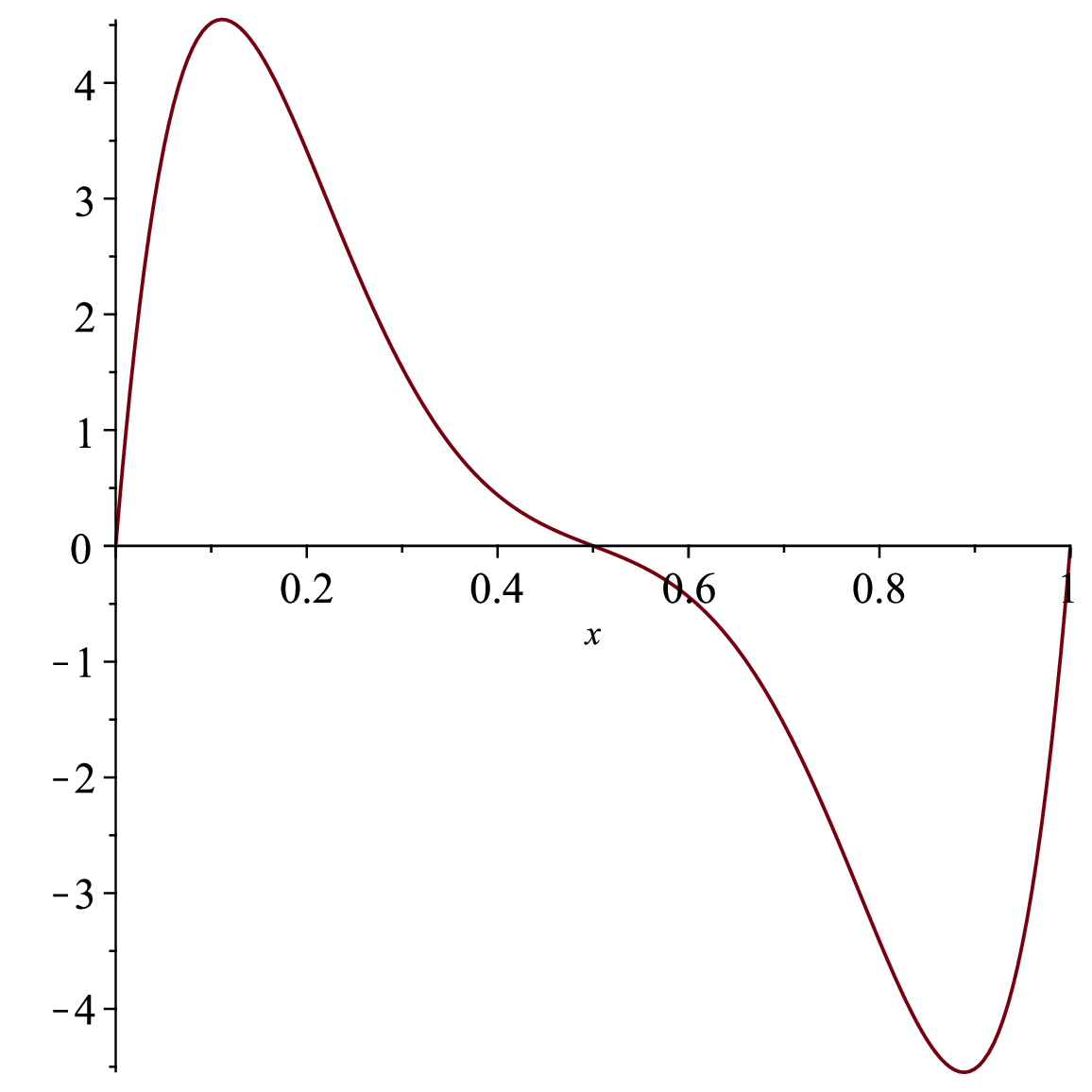} 
\end{center}
\caption{The exact solution of (\ref{de}), (\ref{bc}) and its derivative, with $b(x)=c(x)=f(x)=1$, and $\eps_1 = \eps_2 = 10^{-2}$.}
\label{F1}
\end{figure}

In Figure \ref{F2}, we show the solution and its derivative for $\eps_1 =10^{-3}, \eps_2 = 10^{-1}$. Since $\eps_1 \ll \eps_2^2$, both the solution $u$ and $u'$ feature layers of different widths (as predicted by Theorem \ref{thm:main}).

\begin{figure}[h]
\begin{center}
\includegraphics[width=0.4\textwidth]{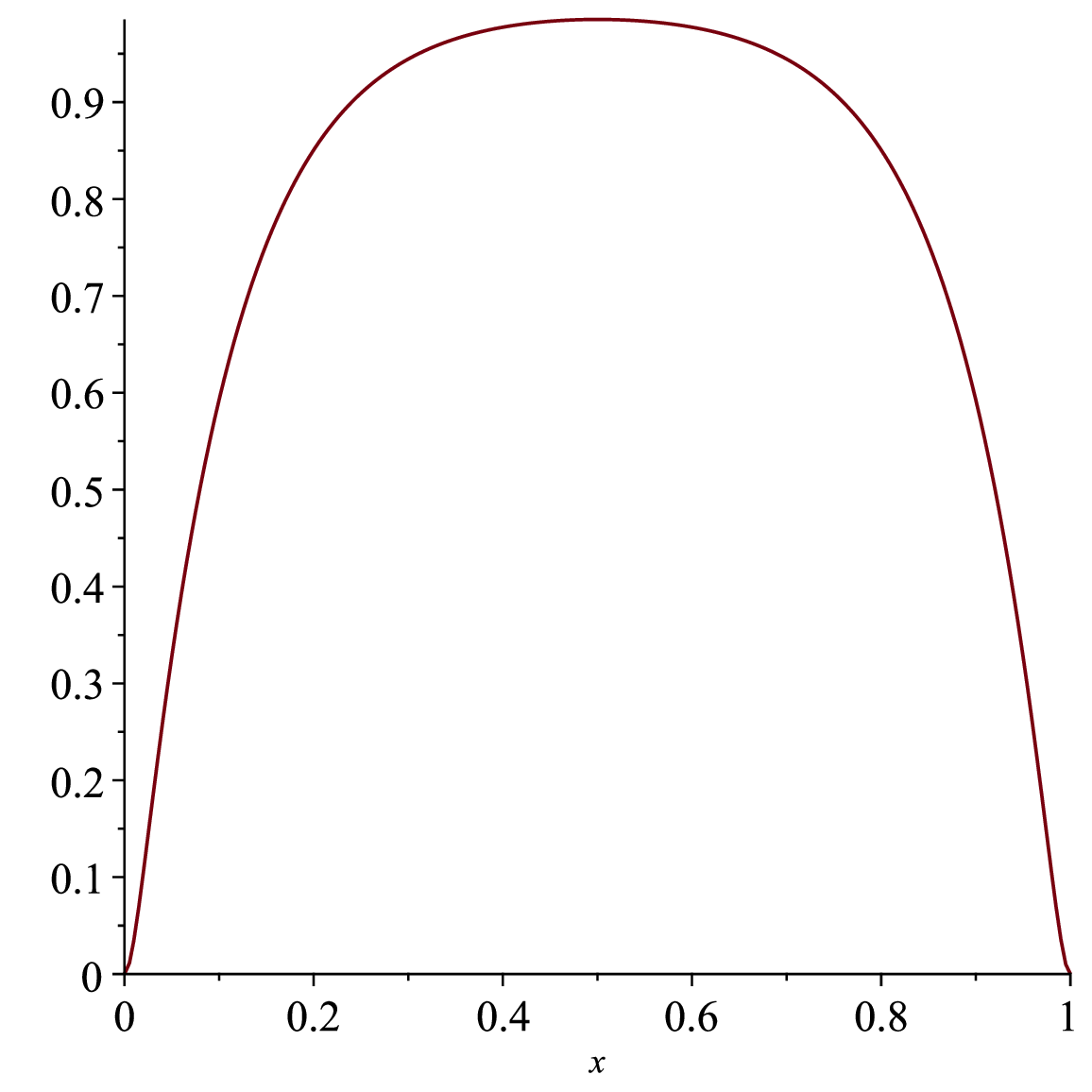} 
\mbox{ }
\includegraphics[width=0.4\textwidth]{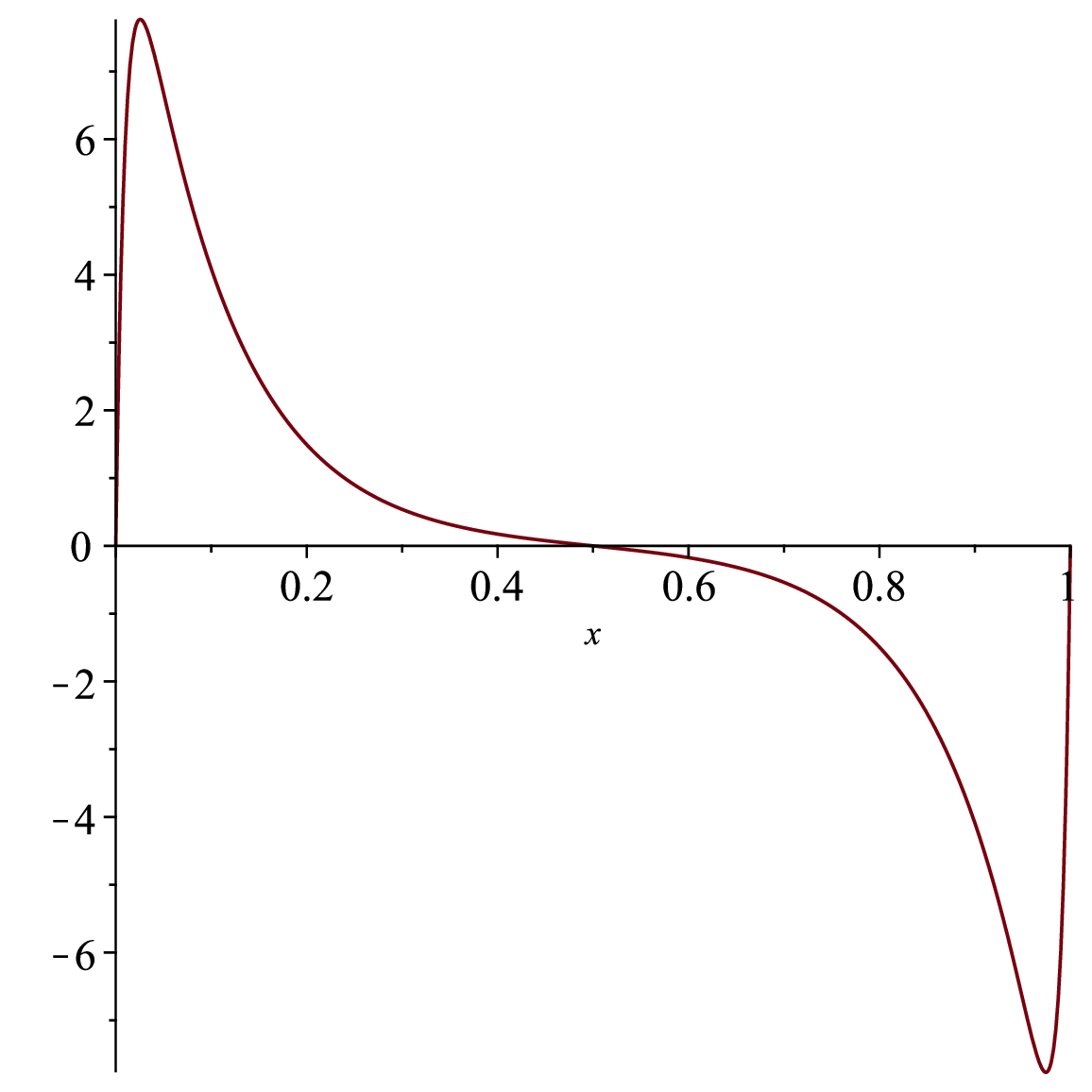} 
\end{center}
\caption{The exact solution of (\ref{de}), (\ref{bc}) and its derivative, with $b(x)=c(x)=f(x)=1$, and $\eps_1 = 10^{-3}, \eps_2 = 10^{-1}$.}
\label{F2}
\end{figure}

Finally, in Figure \ref{F3} we repeat the illustration for $\eps_1 = 10^{-5}, \eps_2 = 10^{-2}$, with the same conclusions as above.

\begin{figure}[h]
\begin{center}
\includegraphics[width=0.4\textwidth]{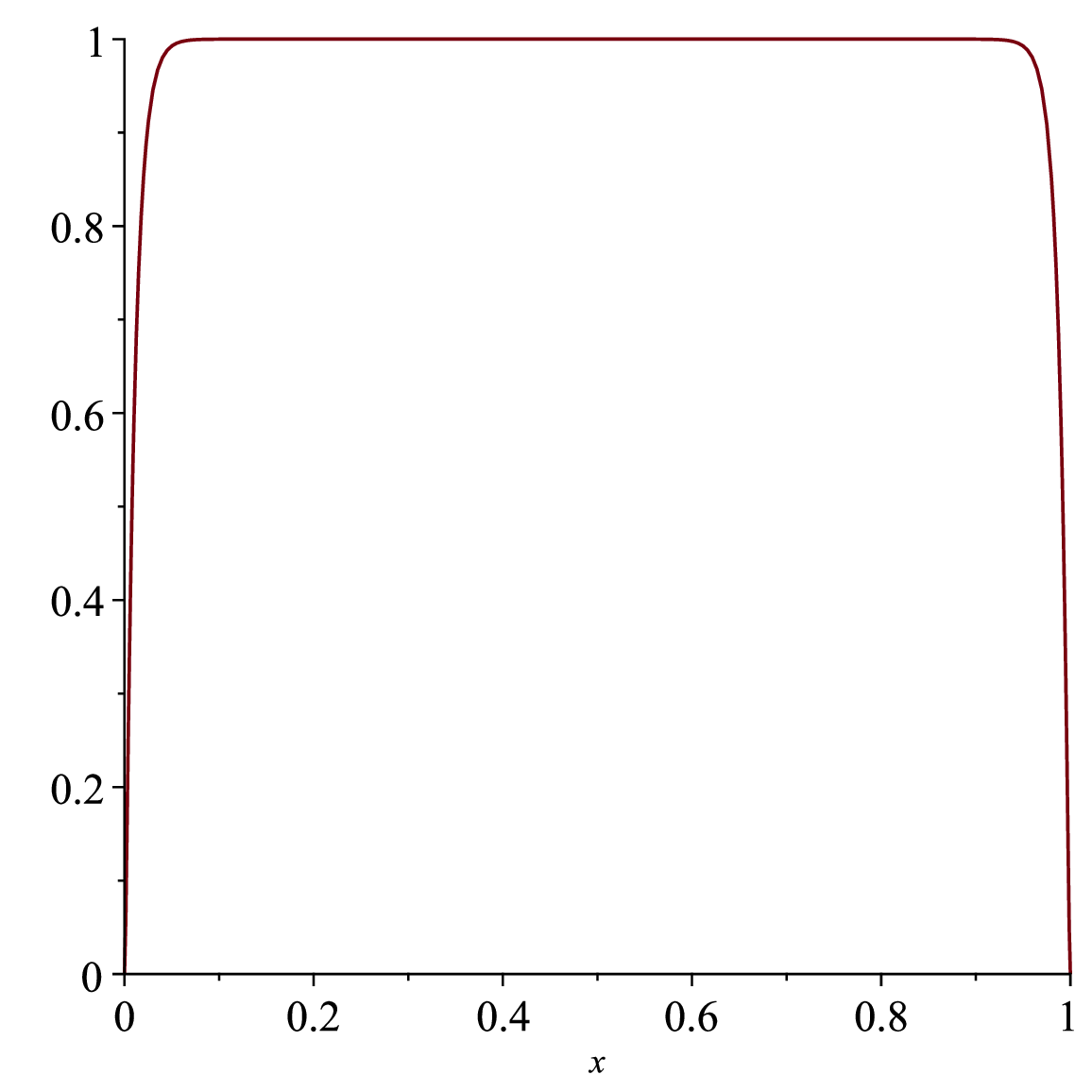} 
\mbox{ }
\includegraphics[width=0.4\textwidth]{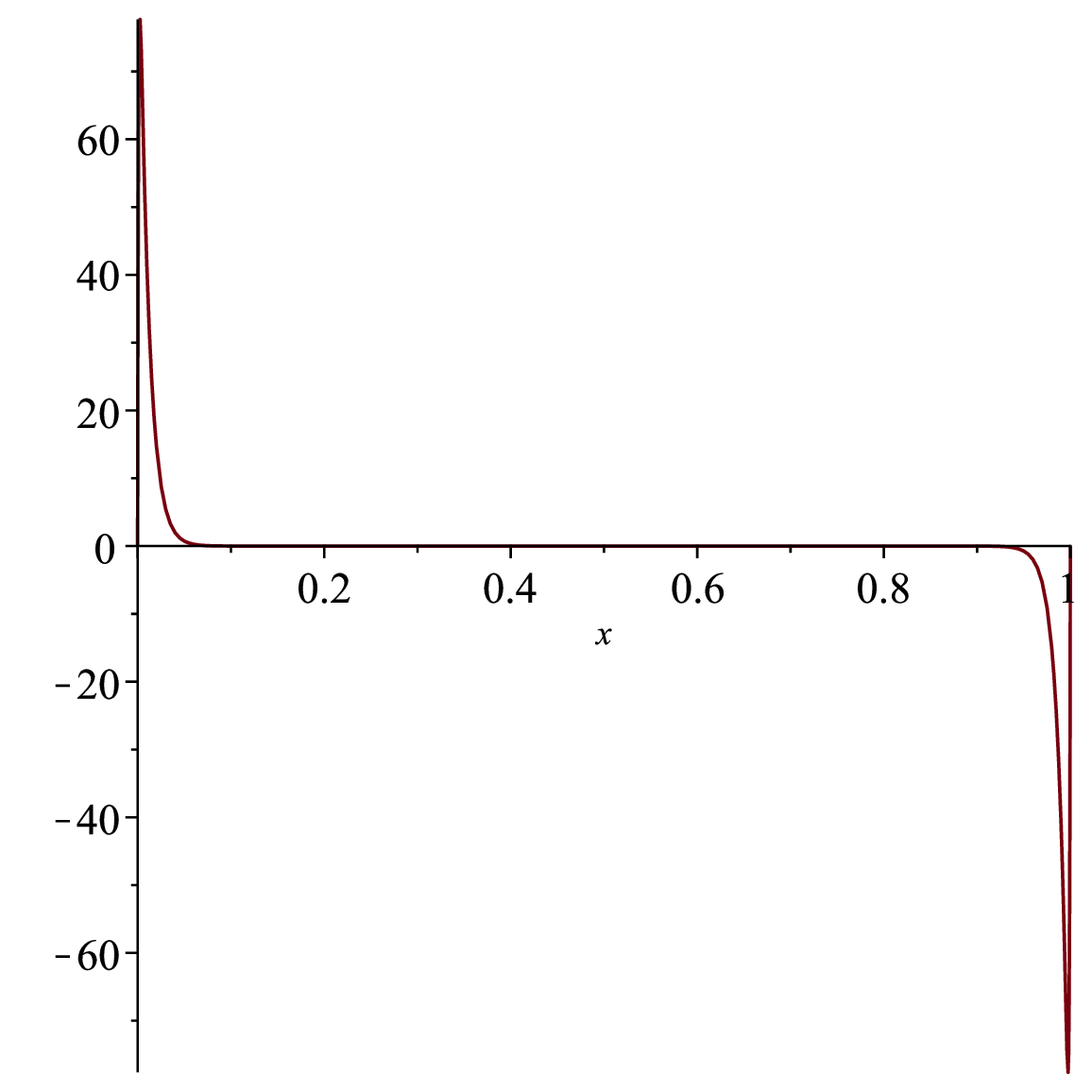} 
\end{center}
\caption{The exact solution of (\ref{de}), (\ref{bc}) and its derivative, with $b(x)=c(x)=f(x)=1$, and $\eps_1 = 10^{-5}, \eps_2 = 10^{-2}$.}
\label{F3}
\end{figure}
%%%%%%%%%%%%%%%%%%%%%%%%%%%%%%%%%%%%%%%%%%%%%%%%%
\section{Conclusions\label{concl}}

In this article we considered a singularly perturbed fourth order boundary value problem, with two
small parameters multiplying the fourth and second derivative, respectively. Under the analyticity of 
the data assumption, we derived classical regularity regults for the solution, showing that it is analytic.
When the parameters are small, the method of matched asymptotic expansions was used to obtain
an expansion for the solution into smooth and layer parts.  Derivative estimates for each part were 
obtained, which show the dependence on the order as well as the perturbation parameters. The main
result of the article gives analytic regularity results for the solution to such problems, necessary for 
proving the convergence of numerical methods for their approximation. In fact, in \cite{Irene} an $hp$ Finite Element Method (FEM) was constructed and analyzed based on the results presented here. The decomposition both guides the construction of a mesh, and facilitates the analysis of the method, yielding \emph{uniform, exponential convergence} as the degree $p$ of the approximating polynomials is increased. (See \cite{Irene} for more details.)

%%%%%%%%%%%%%%%%%%%%%%%%%%%%%%%%%%%%%%%%%%%%%%%%%
\section{Appendix}

Here we provide the details for eqs. (\ref{uij})--(\ref{ubar_ij}).

\subsection{Smooth part}

Beginning with (\ref{uij}), we have 
\begin{equation*}
\sum_{i=0}^{\infty }\sum_{j=0}^{\infty }\varepsilon _{2}^{i}(\varepsilon
_{1}/\varepsilon _{2}^{2})^{j}\left\{ \varepsilon
_{1}^{2}u_{i,j}^{(4)}-\varepsilon _{2}^{2}\left( bu_{i,j}^{\prime }\right)
^{\prime }+cu_{i,j}\right\} =f
\end{equation*}%
\begin{equation*}
\Rightarrow \sum_{i=0}^{\infty }\sum_{j=0}^{\infty }\varepsilon
_{2}^{i}(\varepsilon _{1}^{j+2}/\varepsilon
_{2}^{2j})u_{i,j}^{(4)}-\sum_{i=0}^{\infty }\sum_{j=0}^{\infty }\varepsilon
_{2}^{i+2}(\varepsilon _{1}/\varepsilon _{2}^{2})^{j}\left( bu_{i,j}^{\prime
}\right) ^{\prime }+\sum_{i=0}^{\infty }\sum_{j=0}^{\infty }\varepsilon
_{2}^{i}(\varepsilon _{1}/\varepsilon _{2}^{2})^{j}cu_{i,j}=f.
\end{equation*}%
The first sum is%
\begin{eqnarray*}
\sum_{i=0}^{\infty }\sum_{j=0}^{\infty }\varepsilon _{2}^{i}(\varepsilon
_{1}^{j+2}/\varepsilon _{2}^{2j})u_{i,j}^{(4)} &=&\sum_{i=0}^{\infty
}\sum_{j=2}^{\infty }\varepsilon _{2}^{i}(\varepsilon _{1}^{j}/\varepsilon
_{2}^{2j-4})u_{i,j-2}^{(4)}=\sum_{i=0}^{\infty }\sum_{j=2}^{\infty
}\varepsilon _{2}^{i+4}(\varepsilon _{1}/\varepsilon
_{2}^{2})^{j}u_{i,j-2}^{(4)} \\
&=&\sum_{i=4}^{\infty }\sum_{j=2}^{\infty }\varepsilon _{2}^{i}(\varepsilon
_{1}/\varepsilon _{2}^{2})^{j}u_{i-4,j-2}^{(4)}.
\end{eqnarray*}%
The second sum is%
\begin{equation*}
\sum_{i=0}^{\infty }\sum_{j=0}^{\infty }\varepsilon _{2}^{i+2}(\varepsilon
_{1}/\varepsilon _{2}^{2})^{j}\left( bu_{i,j}^{\prime }\right) ^{\prime
}=\sum_{i=2}^{\infty }\sum_{j=0}^{\infty }\varepsilon _{2}^{i}(\varepsilon
_{1}/\varepsilon _{2}^{2})^{j}\left( bu_{i-2,j}^{\prime }\right) ^{\prime }
\end{equation*}%
\begin{equation*}
=\sum_{i=2}^{\infty }\varepsilon _{2}^{i}\left( bu_{i-2,0}^{\prime
}\right) ^{\prime }+\sum_{i=2}^{\infty }\varepsilon _{2}^{i}(\varepsilon
_{1}/\varepsilon _{2}^{2})\left( bu_{i-2,1}^{\prime }\right) ^{\prime
}+\sum_{j=2}^{\infty }\varepsilon _{2}^{2}(\varepsilon _{1}/\varepsilon
_{2}^{2})^{j}\left( bu_{0,j}^{\prime }\right) ^{\prime }+
\end{equation*}%
\begin{equation*}
+\sum_{j=2}^{\infty
}\varepsilon _{2}^{3}(\varepsilon _{1}/\varepsilon _{2}^{2})^{j}\left(
bu_{1,j}^{\prime }\right) ^{\prime }+\sum_{i=4}^{\infty }\sum_{j=2}^{\infty
}\varepsilon _{2}^{i}(\varepsilon _{1}/\varepsilon _{2}^{2})^{j}\left(
bu_{i-2,j}^{\prime }\right) ^{\prime }.
\end{equation*}%
Finally, the third sum is%
\begin{equation*}
\sum_{i=0}^{\infty }\sum_{j=0}^{\infty }\varepsilon _{2}^{i}(\varepsilon
_{1}/\varepsilon _{2})^{j}cu_{i,j}=cu_{0,0}+\sum_{j=2}^{\infty }(\varepsilon
_{1}/\varepsilon _{2}^{2})^{j}cu_{0,j}+\sum_{j=2}^{\infty }\varepsilon
_{2}(\varepsilon _{1}/\varepsilon _{2}^{2})^{j}cu_{1,j}+\sum_{j=2}^{\infty
}\varepsilon _{2}^{2}(\varepsilon _{1}/\varepsilon _{2}^{2})^{j}cu_{2,j}+
\end{equation*}%
\begin{equation*}
+\sum_{j=2}^{\infty }\varepsilon _{2}^{3}(\varepsilon _{1}/\varepsilon
_{2}^{2})^{j}cu_{3,j}+\sum_{i=1}^{\infty }\varepsilon
_{2}^{i}cu_{i,0}+\sum_{i=0}^{\infty }\varepsilon _{2}^{i}(\varepsilon
_{1}/\varepsilon _{2}^{2})cu_{i,1}+\sum_{i=4}^{\infty }\sum_{j=2}^{\infty
}\varepsilon _{2}^{i}(\varepsilon _{1}/\varepsilon _{2}^{2})^{j}cu_{i,j}.
\end{equation*}%
Immediately, we see that%
\begin{equation*}
u_{0,0}=\frac{f}{c}.
\end{equation*}%
For the rest of the terms we have%
\begin{equation*}
\sum_{i=4}^{\infty }\sum_{j=2}^{\infty }\varepsilon _{2}^{i}(\varepsilon
_{1}/\varepsilon _{2}^{2})^{j}\left\{ u_{i-4,j-2}^{(4)}-\left(
bu_{i-2,j}^{\prime }\right) ^{\prime }+cu_{i,j}\right\} =0,
\end{equation*}%
\begin{equation*}
-\sum_{i=2}^{\infty }\varepsilon _{2}^{i}\left( bu_{i-2,0}^{\prime }\right)
^{\prime }-\sum_{i=2}^{\infty }\varepsilon _{2}^{i}(\varepsilon
_{1}/\varepsilon _{2}^{2})\left( bu_{i-2,1}^{\prime }\right) ^{\prime
}+\sum_{i=1}^{\infty }\varepsilon _{2}^{i}cu_{i,0}+\sum_{i=0}^{\infty
}\varepsilon _{2}^{i}(\varepsilon _{1}/\varepsilon _{2}^{2})cu_{i,1}
=0\Rightarrow \\
\end{equation*}%
\begin{equation*}
\sum_{i=2}^{\infty }\varepsilon _{2}^{i}\left\{ -\left( bu_{i-2,0}^{\prime
}\right) ^{\prime }+cu_{i,0}\right\} +\varepsilon
_{2}cu_{1,0}+\sum_{i=2}^{\infty }\varepsilon _{2}^{i}(\varepsilon
_{1}/\varepsilon _{2}^{2})\left\{ -\left( bu_{i-2,1}^{\prime }\right) ^{\prime
}+cu_{i,1}\right\} +\\
\end{equation*}%
\begin{equation*}
+(\varepsilon _{1}/\varepsilon _{2}^{2})cu_{0,1}+\varepsilon _{2}(\varepsilon _{1}/\varepsilon _{2}^{2})cu_{1,1}=0, \\
\end{equation*}%
\begin{equation*}
\sum_{j=2}^{\infty }(\varepsilon
_{1}/\varepsilon _{2}^{2})^{j}cu_{0,j}+\sum_{j=2}^{\infty }\varepsilon
_{2}(\varepsilon _{1}/\varepsilon _{2}^{2})^{j}cu_{1,j}+\sum_{j=2}^{\infty }\varepsilon _{2}^{2}(\varepsilon _{1}/\varepsilon
_{2}^{2})^{j}\left\{ -\left( bu_{0,j}^{\prime }\right) ^{\prime }+cu_{2,j}\right\} +\\
\end{equation*}%
\begin{equation*}
+\sum_{j=2}^{\infty
}\varepsilon _{2}^{3}(\varepsilon _{1}/\varepsilon _{2}^{2})^{j}\left\{ -\left(
bu_{1,j}^{\prime }\right) ^{\prime }+cu_{3,j}\right\} =0,
\end{equation*}%
from which we see that%
\begin{equation*}
\left. 
\begin{array}{c}
u_{1,0}=0,u_{i,0}=\frac{\left( bu_{i-2,0}^{\prime }\right) ^{\prime }}{c}%
,i\geq 2 \\  
u_{i,1}=0,i\geq 0 \\
u_{0,j}=u_{1,j}=u_{2,j}=u_{3,j}=0,j\geq 2 \\
u_{i,j}=\frac{1}{c}\left\{ \left( bu_{i-2,j}^{\prime }\right) ^{\prime
}-u_{i-4,j-2}^{(4)}\right\} ,i\geq 4,j\geq 2%
\end{array}%
\right\} .
\end{equation*}

\subsection{Boundary layer parts}

For (\ref{ut_ij}) we have 
\begin{equation*}
\sum_{i=0}^{\infty }\sum_{j=0}^{\infty }\varepsilon _{2}^{i}(\varepsilon
_{1}/\varepsilon _{2}^{2})^{j}\left\{ \frac{\varepsilon _{1}^{2}}{%
\varepsilon _{2}^{4}}\left(\tilde{u}^{BL}_{i,j}\right)^{(4)}-\varepsilon _{2}^{2}\left(
\sum_{k=0}^{\infty }\varepsilon _{2}^{k}\tilde{b}_{k}\frac{1}{\varepsilon
_{2}}\left(\tilde{u}^{BL}_{i,j}\right)^{\prime }\right) ^{\prime }+\sum_{k=0}^{\infty }\varepsilon _{2}^{k}\tilde{c%
}_{k}\tilde{u}^{BL}_{i,j}\right\} =0,
\end{equation*}
which leads to
\begin{equation*}
\sum_{i=0}^{\infty }\sum_{j=0}^{\infty }\varepsilon _{2}^{i-4}(\varepsilon
_{1}^{j+2}/\varepsilon _{2}^{2j})\left(\tilde{u}^{BL}_{i,j}\right)^{(4)}-\sum_{i=0}^{\infty
}\sum_{j=0}^{\infty }\varepsilon _{2}^{i}(\varepsilon _{1}/\varepsilon
_{2}^{2})^{j}\sum_{k=0}^{\infty }\varepsilon _{2}^{k}\left[ \tilde{b}_{k}%
\left(\tilde{u}^{BL}_{i,j}\right)^{^{\prime \prime }}+\tilde{b}_{k}^{\prime }\left(\tilde{u}^{BL}%
_{i,j}\right)^{\prime }\right] +
\end{equation*}
\begin{equation*}
+\sum_{i=0}^{\infty }\sum_{j=0}^{\infty }\varepsilon _{2}^{i}(\varepsilon
_{1}/\varepsilon _{2}^{2})^{j}\sum_{k=0}^{\infty }\varepsilon
_{2}^{k}\tilde{c}_{k}\tilde{u}^{BL}_{i,j}=0.
\end{equation*}
The first term above is
\begin{eqnarray*}
\sum_{i=0}^{\infty }\sum_{j=0}^{\infty }\varepsilon _{2}^{i-4}(\varepsilon
_{1}^{j+2}/\varepsilon _{2}^{2j})\left(\tilde{u}^{BL}_{i,j}\right)^{(4)}&=&\sum_{i=0}^{\infty
}\sum_{j=2}^{\infty }\varepsilon _{2}^{i-4}(\varepsilon _{1}^{j}/\varepsilon
_{2}^{2j-4})\left(\tilde{u}^{BL}_{i,j-2}\right)^{(4)} \\
&=&\sum_{i=0}^{\infty }\sum_{j=2}^{\infty
}\varepsilon _{2}^{i}(\varepsilon _{1}/\varepsilon _{2}^{2})^{j}\left(\tilde{u}^{BL}%
_{i,j-2}\right)^{(4)}.
\end{eqnarray*}
The second term is
\begin{equation*}
\sum_{i=0}^{\infty }\sum_{j=0}^{\infty }\varepsilon _{2}^{i}(\varepsilon
_{1}/\varepsilon _{2}^{2})^{j}\sum_{k=0}^{\infty }\varepsilon _{2}^{k}\left[ 
\tilde{b}_{k}\left(\tilde{u}^{BL}_{i,j}\right)^{{\prime \prime }}+\tilde{b}_{k}^{\prime }%
\left(\tilde{u}^{BL}_{i,j}\right)^{\prime }\right] =
\end{equation*}%
\begin{equation*}
=\sum_{i=0}^{\infty }\sum_{j=0}^{\infty }\varepsilon _{2}^{i}(\varepsilon
_{1}/\varepsilon _{2}^{2})^{j}\left\{ \tilde{b}_{0}\left(\tilde{u}^{BL}_{i,j}\right)^{\prime
\prime }+\sum_{k=1}^{\infty }\varepsilon _{2}^{k}\left[ \tilde{b}_{k}\left(\tilde{%
u}^{BL}_{i,j}\right)^{\prime \prime }+\tilde{b}_{k}^{\prime }\left(\tilde{u}^{BL}_{i,j}\right)^{\prime }%
\right] \right\}
\end{equation*}
\begin{equation*}
=\sum_{i=0}^{\infty }\sum_{j=2}^{\infty }\varepsilon _{2}^{i}(\varepsilon
_{1}/\varepsilon _{2}^{2})^{j}\left\{ \tilde{b}_{0}\left(\tilde{u}^{BL}_{i,j}\right)^{\prime
\prime }+\sum_{k=1}^{\infty }\varepsilon _{2}^{k}\left[ \tilde{b}_{k}\left(\tilde{%
u}^{BL}_{i,j}\right)^{\prime \prime }+\tilde{b}_{k}^{\prime }\left(\tilde{u}^{BL}_{i,j}\right)^{\prime }%
\right] \right\} + \\
\end{equation*}
\begin{equation*}
+\sum_{i=0}^{\infty }\varepsilon _{2}^{i}\left\{ \tilde{b}_{0}\left(\tilde{u}^{BL}%
_{i,0}\right)^{\prime \prime }+\sum_{k=1}^{\infty }\varepsilon _{2}^{k}\left[ 
\tilde{b}_{k}\left(\tilde{u}^{BL}_{i,0}\right)^{\prime \prime }+\tilde{b}_{k}^{\prime }%
\left(\tilde{u}^{BL}_{i,0}\right)^{\prime }\right] \right\} + \\
\end{equation*}
\begin{equation*}
+\sum_{i=0}^{\infty }\varepsilon _{2}^{i}(\varepsilon _{1}/\varepsilon
_{2}^{2})\left\{ \tilde{b}_{0}\left(\tilde{u}^{BL}_{i,1}\right)^{\prime \prime
}+\sum_{k=1}^{\infty }\varepsilon _{2}^{k}\left[ \tilde{b}_{k}\left(\tilde{u}^{BL}%
_{i,1}\right)^{\prime \prime}+\tilde{b}_{k}^{\prime }\left(\tilde{u}^{BL}_{i,1}\right)^{\prime }%
\right] \right\} .
\end{equation*}
Finally, the third term is
\begin{equation*}
\sum_{i=0}^{\infty }\sum_{j=0}^{\infty }\varepsilon _{2}^{i}(\varepsilon
_{1}/\varepsilon _{2}^{2})^{j}\sum_{k=0}^{\infty }\varepsilon _{2}^{k}\tilde{%
c}_{k}\tilde{u}^{BL}_{i,j}=\sum_{i=0}^{\infty }\sum_{j=0}^{\infty }\varepsilon
_{2}^{i}(\varepsilon _{1}/\varepsilon _{2}^{2})^{j}\left\{ \tilde{c}_{0}%
\tilde{u}^{BL}_{i,j}+\sum_{k=1}^{\infty }\varepsilon _{2}^{k}\tilde{c}_{k}\tilde{u%
}^{BL}_{i,j}\right\}
\end{equation*}
\begin{equation*}
=\sum_{i=0}^{\infty }\sum_{j=2}^{\infty }\varepsilon _{2}^{i}(\varepsilon
_{1}/\varepsilon _{2}^{2})^{j}\left\{ \tilde{c}_{0}\tilde{u}^{BL}%
_{i,j}+\sum_{k=1}^{\infty }\varepsilon _{2}^{k}\tilde{c}_{k}\tilde{u}^{BL}%
_{i,j}\right\} +\sum_{i=0}^{\infty }\varepsilon _{2}^{i}\left\{ \tilde{c}_{0}%
\tilde{u}^{BL}_{i,0}+\sum_{k=1}^{\infty }\varepsilon _{2}^{k}\tilde{c}_{k}\tilde{u%
}^{BL}_{i,0}\right\} +\\
\end{equation*}
\begin{equation*}
+\sum_{i=0}^{\infty }\varepsilon _{2}^{i}(\varepsilon _{1}/\varepsilon
_{2}^{2})\left\{ \tilde{c}_{0}\tilde{u}^{BL}_{i,1}+\sum_{k=1}^{\infty
}\varepsilon _{2}^{k}\tilde{c}_{k}\tilde{u}^{BL}_{i,1}\right\} .
\end{equation*}
Combining the three, we get
\begin{equation*}
\sum_{i=0}^{\infty }\sum_{j=2}^{\infty }\varepsilon _{2}^{i}(\varepsilon
_{1}/\varepsilon _{2}^{2})^{j}\left\{ \tilde{u}_{i,j-2}^{(4)}-\tilde{b}_{0}%
\left(\tilde{u}^{BL}_{i,j}\right)^{\prime \prime }+\tilde{c}_{0}\tilde{u}^{BL}%
_{i,j}-\sum_{k=1}^{\infty }\varepsilon _{2}^{k}\left[ \tilde{b}_{k}\left(\tilde{u}^{BL}%
_{i,j}\right)^{\prime \prime }+\tilde{b}_{k}^{\prime }\left(\tilde{u}^{BL}_{i,j}\right)^{\prime }%
\right] +\sum_{k=1}^{\infty }\varepsilon _{2}^{k}\tilde{c}_{k}\tilde{u}^{BL}%
_{i,j}\right\} -
\end{equation*}
\begin{equation*}
-\sum_{i=0}^{\infty }\varepsilon _{2}^{i}\left\{ \tilde{b}_{0}\left(\tilde{u}^{BL}%
_{i,0}\right)^{\prime \prime }+\sum_{k=1}^{\infty }\varepsilon _{2}^{k}\left[ 
\tilde{b}_{k}\left(\tilde{u}^{BL}_{i,0}\right)^{\prime \prime }+\tilde{b}_{k}^{\prime }%
\left(\tilde{u}^{BL}_{i,0}\right)^{\prime }\right] \right\} 
\end{equation*}
\begin{equation*}
-\sum_{i=0}^{\infty }\varepsilon
_{2}^{i}(\varepsilon _{1}/\varepsilon _{2}^{2})\left\{ \tilde{b}_{0}\left(\tilde{u}^{BL}%
_{i,1}\right)^{\prime \prime }+\sum_{k=1}^{\infty }\varepsilon _{2}^{k}\left[ 
\tilde{b}_{k}\left(\tilde{u}^{BL}_{i,1}\right)^{\prime \prime }+\tilde{b}_{k}^{\prime }%
\left(\tilde{u}^{BL}_{i,1}\right)^{\prime }\right] \right\} +
\end{equation*}
\begin{equation*}
+\sum_{i=0}^{\infty }\varepsilon _{2}^{i}\left\{ \tilde{c}_{0}\tilde{u}^{BL}%
_{i,0}+\sum_{k=1}^{\infty }\varepsilon _{2}^{k}\tilde{c}_{k}\tilde{u}^{BL}%
_{i,0}\right\} +\sum_{i=0}^{\infty }\varepsilon _{2}^{i}(\varepsilon
_{1}/\varepsilon _{2}^{2})\left\{ \tilde{c}_{0}\tilde{u}^{BL}_{i,1}+\sum_{k=1}^{%
\infty }\varepsilon _{2}^{k}\tilde{c}_{k}\tilde{u}^{BL}_{i,1}\right\} =0,
\end{equation*}
which in turn gives
\begin{equation*}
\sum_{i=0}^{\infty }\sum_{j=2}^{\infty }\varepsilon _{2}^{i}(\varepsilon
_{1}/\varepsilon _{2}^{2})^{j}\left\{ \left(\tilde{u}^{BL}_{i,j-2}\right)^{(4)}-\tilde{b}_{0}%
\left(\tilde{u}^{BL}_{i,j}\right)^{\prime \prime }+\tilde{c}_{0}\tilde{u}^{BL}%
_{i,j}-\sum_{k=1}^{\infty }\varepsilon _{2}^{k}\left[ \tilde{b}_{k}\left(\tilde{u}^{BL}%
_{i,j}\right)^{\prime \prime}+\tilde{b}_{k}^{\prime }\left(\tilde{u}^{BL}_{i,j}\right)^{\prime }%
\right] \right. 
\end{equation*}
\begin{equation*}
+\left. \sum_{k=1}^{\infty }\varepsilon _{2}^{k}\tilde{c}_{k}\tilde{u}^{BL}%
_{i,j} \right\}+\sum_{i=0}^{\infty }\varepsilon _{2}^{i}\left\{ \tilde{c}_{0}\tilde{u}^{BL}%
_{i,0}-\tilde{b}_{0}\left(\tilde{u}^{BL}_{i,0}\right)^{\prime \prime }+\sum_{k=1}^{\infty
}\varepsilon _{2}^{k}\left\{ \tilde{c}_{k}\tilde{u}^{BL}_{i,0}-\left[ \tilde{b}%
_{k}\left(\tilde{u}^{BL}_{i,0}\right)^{\prime \prime }+\tilde{b}_{k}^{\prime }\left(\tilde{u}^{BL}%
_{i,0}\right)^{\prime }\right] \right\} \right\} +
\end{equation*}
\begin{equation*}
+\sum_{i=0}^{\infty }\varepsilon _{2}^{i}(\varepsilon _{1}/\varepsilon
_{2}^{2})\left\{ \tilde{c}_{0}\tilde{u}^{BL}_{i,1}-\tilde{b}_{0}\left(\tilde{u}^{BL}%
_{i,1}\right)^{\prime \prime }+\sum_{k=1}^{\infty }\varepsilon _{2}^{k}\left\{ 
\tilde{c}_{k}\tilde{u}^{BL}_{i,1}-\left[ \tilde{b}_{k}\left(\tilde{u}^{BL}_{i,1}\right)^{\prime
\prime }+\tilde{b}_{k}^{\prime }\left(\tilde{u}^{BL}_{i,1}\right)^{\prime }\right] \right\}
\right\} =0.
\end{equation*}
We arrive at
\begin{equation*}
\left. 
\begin{array}{c}
-\tilde{b}_{0}\left(\tilde{u}^{BL}_{0,0}\right)^{\prime \prime }+\tilde{c}_{0}\tilde{u}^{BL}_{0,0}=0 \\
-\tilde{b}_{0}\left(\tilde{u}^{BL}_{i,0}\right)^{\prime \prime }+\tilde{c}_{0}\tilde{u}^{BL}_{i,0}=%
\sum_{\ell =1}^{i}\left\{ \left( \tilde{b}_{\ell }%
\left(\tilde{u}^{BL}_{i-\ell ,0}\right)^{\prime }\right) ^{\prime }-\tilde{c}_{\ell }\tilde{u}^{BL}%
_{i-\ell ,0}\right\} ,i\geq 1 \\
-\tilde{b}_{0}\left(\tilde{u}^{BL}_{0,1}\right)^{\prime \prime }+\tilde{c}_{0}\tilde{u%
}^{BL}_{0,1}=0 \\
-\tilde{b}_{0}\left(\tilde{u}^{BL}_{i,1}\right)^{\prime \prime }+\tilde{c}_{0}\tilde{u}^{BL}_{i,1}=%
\sum_{\ell =1}^{i}\left\{ \left( \tilde{b}_{\ell }%
\left(\tilde{u}^{BL}_{i-\ell ,1}\right)^{\prime }\right) ^{\prime }-\tilde{c}_{\ell }\tilde{u}^{BL}%
_{i-\ell ,1}\right\} ,i\geq 1 \\
-\tilde{b}_{0}\left(\tilde{u}^{BL}_{0,j}\right)^{\prime \prime }+\tilde{c}_{0}\tilde{u}^{BL}_{0,j}=-%
\left(\tilde{u}^{BL}_{0,j-2}\right)^{(4)},j\geq 2 \\
-\tilde{b}_{0}\left(\tilde{u}^{BL}_{i,j}\right)^{\prime \prime }+\tilde{c}_{0}\tilde{u}^{BL}_{i,j}=-%
\left(\tilde{u}^{BL}_{i,j-2}\right)^{(4)}+\sum_{\ell =1}^{i}\left\{ \left( \tilde{b}_{\ell }%
\left(\tilde{u}^{BL}_{i-\ell ,j}\right)^{\prime }\right) ^{\prime }-\tilde{c}_{\ell }\tilde{u}^{BL}%
_{i-\ell ,j}\right\} ,i\geq 1,j\geq 2%
\end{array}%
\right\}
\end{equation*}
as desired.

Next we show (\ref{ubar_ij}): we begin with
\begin{equation*}
\sum_{i=0}^{\infty }\sum_{j=0}^{\infty }\varepsilon _{2}^{i}(\varepsilon
_{1}/\varepsilon _{2}^{2})^{j}\left\{ \varepsilon _{1}^{2}\left( \frac{\varepsilon
_{1}}{\varepsilon _{2}}\right) ^{-4}\left(\bar{u}^{BL}_{i,j}\right)^{(4)}-\varepsilon
_{2}^{2}\left( \sum_{k=0}^{\infty }\left( \frac{\varepsilon _{1}}{\varepsilon _{2}}%
\right) ^{k} \bar{b}_{k}\left( \frac{\varepsilon
_{1}}{\varepsilon _{2}}\right) ^{-1}\left(\bar{u}^{BL}_{i,j}\right)^{\prime }\right)
^{\prime } \right.
\end{equation*}
\begin{equation*}
\left. +\sum_{k=0}^{\infty }\left( \frac{\varepsilon _{1}}{\varepsilon
_{2}}\right) ^{k}\bar{c}_{k}\bar{u}^{BL}_{i,j}\right\} =0,
\end{equation*}
or equivalently
\begin{equation*}
\sum_{i=0}^{\infty }\sum_{j=0}^{\infty }\varepsilon _{2}^{i}(\varepsilon
_{1}/\varepsilon _{2}^{2})^{j-2}\left(\bar{u}^{BL}_{i,j}\right)^{(4)} 
\end{equation*}
\begin{equation*}
-\sum_{i=0}^{\infty
}\sum_{j=0}^{\infty }\varepsilon _{2}^{i}(\varepsilon _{1}/\varepsilon
_{2}^{2})^{j-2}\sum_{k=0}^{\infty }\left( \frac{\varepsilon _{1}}{\varepsilon _{2}}%
\right) ^{k}\left[ \left(\bar{b}_{k}\bar{u}^{BL}_{i,j}\right)^{\prime \prime }+\bar{b}_{k}^{\prime }\left(\bar{u}^{BL}_{i,j}\right)^{\prime }\right]
\end{equation*}
\begin{equation*}
+\sum_{i=0}^{\infty }\sum_{j=0}^{\infty }\varepsilon _{2}^{i}(\varepsilon
_{1}/\varepsilon _{2}^{2})^{j}\sum_{k=0}^{\infty }\left( \frac{\varepsilon
_{1}}{\varepsilon _{2}}\right) ^{k}\bar{c}_{k}\bar{u}^{BL}_{i,j}=0 ,
\end{equation*}
which we write as
\begin{equation*}
\sum_{i=0}^{\infty }\varepsilon _{2}^{i}(\varepsilon
_{1}/\varepsilon _{2}^{2})^{-2}\left(\bar{u}^{BL}_{i,0}\right)^{(4)}+\sum_{i=0}^{\infty }\varepsilon _{2}^{i}(\varepsilon
_{1}/\varepsilon _{2}^{2})^{-1}\left(\bar{u}^{BL}_{i,1}\right)^{(4)}+\sum_{i=0}^{\infty }\sum_{j=0}^{\infty }\varepsilon _{2}^{i}(\varepsilon
_{1}/\varepsilon _{2}^{2})^{j}\left(\bar{u}^{BL}_{i,j+2}\right)^{(4)}
\end{equation*}
\begin{equation*}-\sum_{i=0}^{\infty} \varepsilon _{2}^{i}(\varepsilon _{1}/\varepsilon
_{2}^{2})^{-2}\sum_{k=0}^{\infty }\left( \frac{\varepsilon _{1}}{\varepsilon _{2}}%
\right) ^{k}\left[ \bar{b}_{k}\left(\bar{u}^{BL}_{i,0}\right)^{\prime \prime }+\bar{b}_{k}^{\prime }\left(\bar{u}^{BL}_{i,0}\right)^{\prime }\right]
\end{equation*}
\begin{equation*}-\sum_{i=0}^{\infty} \varepsilon _{2}^{i}(\varepsilon _{1}/\varepsilon
_{2}^{2})^{-1}\sum_{k=0}^{\infty }\left( \frac{\varepsilon _{1}}{\varepsilon _{2}}%
\right) ^{k}\left[ \bar{b}_{k}\left(\bar{u}^{BL}_{i,1}\right)^{\prime \prime }+\bar{b}_{k}^{\prime }\left(\bar{u}^{BL}_{i,1}\right)^{\prime }\right]
\end{equation*}
\begin{equation*}-\sum_{i=0}^{\infty
}\sum_{j=0}^{\infty }\varepsilon _{2}^{i}(\varepsilon _{1}/\varepsilon
_{2}^{2})^{j}\sum_{k=0}^{\infty }\left( \frac{\varepsilon _{1}}{\varepsilon _{2}}%
\right) ^{k}\left[ \bar{b}_{k}\left(\bar{u}^{BL}_{i,j+2}\right)^{\prime \prime }+\bar{b}_{k}^{\prime }\left(\bar{u}^{BL}_{i,j+2}\right)^{\prime }\right]
\end{equation*}%
\begin{equation*}
+\sum_{i=0}^{\infty }\sum_{j=0}^{\infty }\varepsilon _{2}^{i}(\varepsilon
_{1}/\varepsilon _{2}^{2})^{j}\sum_{k=0}^{\infty }\left( \frac{\varepsilon
_{1}}{\varepsilon _{2}}\right) ^{k}\bar{c}_{k}\bar{u}^{BL}_{i,j}=0.
\end{equation*}
This gives
\begin{equation*}
\sum_{i=0}^{\infty }\sum_{j=0}^{\infty }\varepsilon _{2}^{i}(\varepsilon
_{1}/\varepsilon _{2}^{2})^{j}\left\{\left( \bar{u}^{BL}_{i,j+2}\right)^{(4)}-\bar{b}_{0}\left(\bar{u}^{BL}_{i,j+2}\right)^{\prime \prime }+\bar{c}_{0}\bar{u}^{BL}_{i,j} \right.
\end{equation*}
\begin{equation*}
\left. +\sum_{k=1}^{\infty }\varepsilon _{2}^{k}\left( \frac{\varepsilon _{1}}{\varepsilon _{2}^{2}}%
\right) ^{k}\left\{ \bar{c}_{k}\bar{u}^{BL}_{i,j}-\left[ \bar{b}_{k}\left(\bar{u}^{BL}_{i,j+2}\right)^{\prime \prime }+\bar{b}_{k}^{\prime }\left(\bar{u}^{BL}_{i,j+2}\right)^{\prime }\right] \right\} \right\}
\end{equation*}
\begin{equation*}
+\sum_{i=0}^{\infty }\varepsilon _{2}^{i}(\varepsilon _{1}/\varepsilon
_{2}^{2})^{-2}\left\{ \left(\bar{u}^{BL}_{i,0}\right)^{(4)}-\bar{b}_{0}\left(\bar{u}^{BL}_{i,0}\right)^{\prime \prime }-\sum_{k=1}^{\infty }\varepsilon _{2}^{k}\left( \frac{%
\varepsilon _{1}}{\varepsilon _{2}^{2}}\right) ^{k}\left[ \bar{b}_{k}\left(\bar{u}^{BL}_{i,0}\right)^{\prime \prime }+\bar{b}_{k}^{\prime }\left(\bar{u}^{BL}_{i,0}\right)^{\prime }\right] \right\}
\end{equation*}
\begin{equation*}
+\sum_{i=0}^{\infty
}\varepsilon _{2}^{i}(\varepsilon _{1}/\varepsilon _{2}^{2})^{-1}\left\{ 
\bar{u}_{i,1}^{(4)}-\bar{b}_{0}\left(\bar{u}^{BL}_{i,1}\right)^{\prime \prime }-\sum_{k=1}^{\infty }\varepsilon _{2}^{k}\left( \frac{%
\varepsilon _{1}}{\varepsilon _{2}^{2}}\right) ^{k}\left[ \bar{b}_{k}\left(\bar{u}^{BL}_{i,1}\right)^{\prime \prime }+\bar{b}_{k}^{\prime }\left(\bar{u}^{BL}_{i,1}\right)^{\prime }\right] \right\} =0,
\end{equation*}
from which we obtain the desired result.

%---------------------------------------

\end{document}